\documentclass[twoside,11pt]{article}

\usepackage{jmlr2e1}

\usepackage{amsmath}
\usepackage{enumerate}
\usepackage{color}

\newcommand{\new}{\newcommand}

\providecommand{\nor}[1]{\lVert{#1}\rVert}

\providecommand{\abs}[1]{\lvert{#1}\rvert}
\providecommand{\set}[1]{\{#1\}}
\providecommand{\scal}[2]{\langle{#1},{#2}\rangle}

\providecommand{\ker}[1]{\operatorname{ker}(#1)} 
\providecommand{\ran}[1]{\operatorname{ran}#1}
\providecommand{\dim}[1]{\operatorname{dim}(#1)}

\providecommand{\lspan}[1]{\operatorname{span}\{#1\}}

\providecommand{\tr}[1]{\operatorname{Tr}(#1)}

\providecommand{\wh}[1]{\widehat{#1}_n}

\new{\R}{\mathbb R}
\new{\CC}{\mathbb C}
\new{\N}{\mathbb N}
\new{\I}{\mathbb I}

\new{\hh}{\mathcal H}
\new{\kk}{\mathcal K}
\new{\cT}{\mathcal T}
\new{\sS}{\mathcal S}

\new{\la}{\lambda}
\new{\eps}{\epsilon}
\new{\ga}{\gamma}

\new{\red}[1]{#1}
\new{\blue}[1]{#1}

\ShortHeadings{A Machine Learning Approach to Optimal Tikhonov
  Regularization}{De Vito, Fornasier and Naumova} 
\firstpageno{1}

\begin{document}

\title{A Machine Learning Approach to Optimal Tikhonov Regularization
  I: Affine Manifolds }

\author{\name{Ernesto De Vito ({\em corresponding author})}
\email{devito@dima.unige.it} \\
\addr{DIMA, Universit\`a di Genova,  Via
  Dodecaneso 35, Genova,   Italy }
\AND
\name{Massimo Fornasier} \email{massimo.fornasier@ma.tum.de} \\
\addr{Technische Universit\"at M\"unchen, Fakult\"at Mathematik,
  Boltzmannstrasse 3 D-85748,\\ Garching bei M\"unchen, Germany }
\AND
\name{Valeriya Naumova}\email{valeriya@simula.no} \\
\addr{Simula Research Laboratory, Martin Linges vei 25, Fornebu, Norway}
}

\editor{}

\maketitle

\begin{abstract}
Despite a variety of available techniques the issue of the proper regularization parameter choice for inverse problems still remains one of the relevant challenges in the field.  The main difficulty lies in constructing a rule, allowing to compute the parameter from given noisy data without relying either on any a priori knowledge of the solution or on the noise level. In this paper we propose a novel method based on supervised machine learning to approximate the high-dimensional function, mapping noisy data into a good approximation to the optimal Tikhonov regularization parameter. Our assumptions are that solutions of the inverse problem are statistically distributed in a concentrated manner on (lower-dimensional) linear subspaces and the noise is sub-gaussian. We show that the number of previously observed examples for the supervised learning of  the optimal parameter mapping scales at most linearly with the dimension of the solution subspace. Then we also provide explicit error bounds on the accuracy of the approximated parameter and the corresponding regularization solution.  
Even though the results are more of theoretical nature, we present a recipe for the practical implementation of the approach, we discuss its computational complexity, and provide numerical experiments confirming the theoretical results. We also outline interesting directions for future research  with some preliminary results, confirming their feasibility. 
\end{abstract}

\begin{keywords}
Tikhonov regularization, parameter choice rule, sub-gaussian vectors,
high dimensional function approximations, concentration inequalities. 

\end{keywords}

\section{Introduction}

In many practical problems, one cannot observe directly the quantities
of most interest; instead their values have to be inferred from their effects
on observable quantities. When this relationship between observable
$Y$ and the quantity of interest $X$ is (approximately) linear, as it is in surprisingly many
cases, the situation can be modeled mathematically by the equation
\begin{equation}\label{eq:11}
Y= A X
\end{equation}
for $A$ being a linear operator model. If $A$ is a ``nice'', easily invertible operator, and if the data
$Y$ are noiseless and complete, then finding $X$ is a trivial task. Often, however, the mapping $A$
is ill-conditioned or not invertible. Moreover, typically \eqref{eq:11} is only an
idealized version, which completely neglects any presence of noise or disturbances; a more accurate model
is
\begin{equation}\label{eq:12}
Y= A X+ \eta,
\end{equation}
in which the data are corrupted by an (unknown) noise.  In order to deal with
this type of reconstruction problem a regularization
mechanism is required \citep{zbMATH00936298}. 

Regularization techniques attempt to incorporate as much as possible an (often vague) a priori knowledge on the nature of the solution $X$.  A well-known assumption which is often
used to regularize inverse problems is that the solution belongs to some ball of a suitable Banach space.

Regularization theory has shown to play its major role for solving infinite dimensional inverse problems. In this paper, however,
we consider finite dimensional problems, since we intend to use probabilistic techniques 
for which the Euclidean space is the most standard setting. Accordingly, we assume the solution vector $X \in \mathbb R^d$, the linear model $A \in \mathbb R^{m\times d}$, and the  datum $Y \in \mathbb R^m$. In the following we denote with $\|Z\|$ the Euclidean norm of a vector $Z \in  \mathbb R^N$. One of the most  widely used regularization approaches is realized by minimizing the following, so-called, Tikhonov functional 
\begin{equation}\label{eq:13}
\min_{z\in\R^d} \nor{Az-Y}^2 + \alpha\, \nor{z}^2. 
\end{equation}
with $\alpha \in (0,+\infty)$. The {\it regularized solution} $Z^\alpha:=Z^\alpha(Y)$ of such minimization procedure is unique. In this context, the regularization scheme represents a trade-off between the accuracy of fitting the data $Y$ and the complexity of the solution, measured by a 
ball in $\mathbb R^d$ with radius  depending on the {\it regularization parameter} $\alpha$. 
Therefore, the choice of the regularization parameter $\alpha$ is very crucial to identify the best possible regularized solution, which does not overfit the noise. This issue still remains one of the most delicate aspects of this approach and other regularization schemes.  Clearly the best possible parameter minimizes the discrepancy between $Z^\alpha$ and the solution $X$
$$
\alpha^* = \arg \min_{\alpha \in (0,+\infty)} \| Z^\alpha -X \|.
$$
Unfortunately, we usually have neither access to the solution $X$ nor to information about the noise, for instance, we might not be aware of the noise level $\|\eta\|$. Hence, for determining a possible good approximation to the optimal regularization parameter several approaches have been proposed, which can be categorized into three classes
\begin{itemize}
\item A priori parameter choice  rules based on the noise level and some known ``smoothness'' of the solution encoded in terms, e.g., of the so-called {\it source condition} \citep{zbMATH00936298};
\item A posteriori parameter  choice rules based on the datum $Y$ and the noise level;
\item A posteriori parameter  choice rules based exclusively on the datum $Y$ or, the so-called, heuristic parameter choice rules.
\end{itemize}
For the latter two categories there are by now a multitude of approaches. Below we recall  the most used and relevant of them, indicating in square brackets their alternative names, accepted in different scientific communities. In most cases, the names we provide are the descriptive names originally given to the methods. However, in a few cases, there was no original name, and, to achieve consistency in the naming, we have chosen an appropriate one, reflecting the nature of the method. We mention, for instance,
(transformed/modified) discrepancy principle [Raus-Gfrerer rule, minimum bound method]; monotone error rule; (fast/hardened) balancing principle also for white noise; quasi-optimality criterion; L-curve method; modified discrepancy partner rule [Hanke-Raus rule]; extrapolated error method; normalized cumulative periodogram method; residual method; generalized maximum likelihood; (robust/strong robust/modified) generalized cross-validation. Considering the large number of available parameter choice methods, there are relatively few comparative studies and we refer to \citep{zbMATH05929140} for a rather comprehensive discussion on their differences, pros and contra. One of the features which is common to most of the a posteriori parameter choice rules is the need of solving \eqref{eq:13} multiple times for different values of the parameters $\alpha$, often selected out of a conveniently pre-defined grid.
\\

In this paper, we intend to study a novel, {data-driven, regularization method, which also yields approximations to the  optimal parameter in Tikhonov regularization.}  After an off-line learning phase, whose complexity scales at most algebraically with the dimensionality of the problem, our method does not require any additional knowledge of the noise level; the computation of a near-optimal regularization parameter can be performed very efficiently { by solving the regularization problem \eqref{eq:13} only a moderated amount times, see Section \ref{sec:discussion} for a discussion on the computational complexity. In particular cases, no solution of  \eqref{eq:13} is actually needed, see Section \ref{sec:5}. Not being based on the noise level, our approach  fits into the class of {\it heuristic parameter choice rules} \citep{zbMATH06288322}.}
The approach aims at employing the framework of supervised machine learning to the problem of approximating the high-dimensional function, which maps noisy data into the corresponding optimal regularization parameter. More precisely, we assume that we are allowed to see a certain number $n$ of examples of solutions $X_i$ and corresponding noisy data $Y_i = A X_i + \eta_i$, for $i=1,\dots,n$. For all of these examples, we are clearly capable to compute the optimal regularization parameters as in the following scheme
\begin{eqnarray*}
(X_1,Y_1) &\to& \alpha_1^*= \arg \min_{\alpha \in (0,+\infty)} \| Z^\alpha(Y_1) -X_1 \|\\
(X_2,Y_2) &\to& \alpha_2^*= \arg \min_{\alpha \in (0,+\infty)} \| Z^\alpha(Y_2) -X_2 \| \\
\dots && \dots \\
(X_n,Y_n) &\to& \alpha_n^*= \arg \min_{\alpha \in (0,+\infty)} \| Z^\alpha(Y_n) -X_n \|\\
(?? ,Y) &\to&\bar \alpha
\end{eqnarray*}
Denote $\mu$ the joint distribution of the empirical samples $(Y_1, \alpha^*_1), \dots, (Y_n, \alpha^*_n)$. Were  its conditional distribution $\mu (\cdot\mid Y)$ with respect to the first variable $Y$ very much concentrated (for instance, when $\int_0^\infty (\alpha - \bar \alpha)^q d \mu (\alpha \mid Y)$ is very small for $q \geq 1$ and for variable $Y$), then we could design a proper regression function $$\mathcal R: Y \mapsto \bar \alpha:= \mathcal R(Y) = \int_0^\infty \alpha d \mu(\alpha \mid Y).$$
Such a mapping would allow us, to a given new datum $Y$ (without given solution!), to associate the corresponding  parameter $\bar \alpha$ not too far from the true optimal one $\alpha^*$, at least with high probability. We illustrate schematically this theoretical framework in Figure \ref{introfig}.

\begin{figure}[h]
  \centering\
  \includegraphics[width=0.6\textwidth]{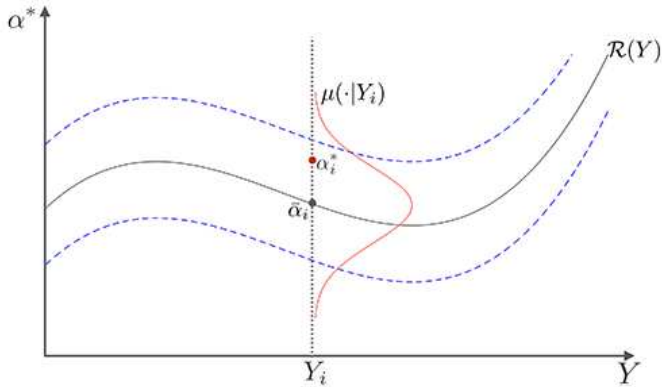}   
  \caption{Learning optimal regularization parameters from previously observed samples by approximation of the regression function $\mathcal R$.}
\label{introfig}
\end{figure}

At a first glance, this setting may seem quite hopeless. First of all, one should establish the concentration of the conditional distribution generating $\alpha^*$ given $Y$. Secondly, even if we  assume that the regression function $\mathcal R$ is very smooth, the vectors $Y$ belong to the space $\mathbb R^m$ and the number of observations $n$ required to learn such a function need to scale exponentially with the dimension $m$ \citep{zbMATH05659507}.
It is clear that we cannot address neither of the above issues in general. The only hope is that the solutions are statistically distributed in a concentrated manner over smooth sets of lower dimension $h \ll m$ and the noise has also a concentrated distribution, so that the corresponding data $Y$ are concentrated around lower-dimensional sets as well. And luckily these two assumptions are to a certain extent realistic.

By now, the assumption that the  possible solutions belong to a lower-dimensional set of $\mathbb R^d$ has become an important prior for many signal and image processing tasks. For instance, were solutions natural images, then it is known that images can be represented as nearly-sparse coefficient vectors with respect to shearlets expansions \citep{zbMATH05993957}. { Hence, in this case the set of possible solutions can be stylized as a union of lower-dimensional linear subspaces, consisting of {\it sparse} vectors \citep{zbMATH05570252}}. In other situations, it is known that the solution set can be stylized, at least locally, { as a smooth lower-dimensional nonlinear manifold $\mathcal V$ \citep{zbMATH06031284, zbMATH06454789, zbMATH06770640}}. Also in this case, at least locally, it is possible to approximate the solution set by means of affine lower-dimensional sets, representing tangent spaces to the manifold. 
Hence, the a priori knowledge that the solution is belonging to some special (often nonlinear) set should also be taken into account when designing the regularization method.
\\

In this paper, we want to show very rigorously how  one can construct, from a relatively small number of previously observed examples, an approximation $\widehat{\mathcal{R}}$ to the regression function $\mathcal R$, which is mapping a noisy datum into a good approximation to the optimal Tikhonov regularization parameter. To this end, we assume the solutions to be distributed sub-gaussianly over a  {\it linear subspace} $\mathcal V \subset \mathbb R^d$ of dimension $h \ll m$ and the noise $\eta$ to be also sub-gaussian. 
The first statistical assumption is perhaps mostly technical to allow us to provide rigorous estimates.
Let us describe the method of computation as follows.
We introduce the $\red{m\times m}$ noncentered covariance matrix built from the noisy measurements 
\begin{equation*}
\wh{\Sigma}= \frac{1}{n} \sum_{i=1}^n Y_i\otimes Y_i,\label{eq:60}
\end{equation*}
and we denote by $\wh{\Pi}$   the projections onto the
vector space spanned by the first most relevant eigenvectors of $\wh{\Sigma}$.
Furthermore, we set
$\wh{\alpha}\in (0,+\infty)$ as the minimizer of 
\[
\min_{\alpha\in (0,+\infty)} \nor{Z^\alpha-A^{\dagger} \wh{\Pi} Y}^2,
\]
where $A^{\dagger}$ is the pseudo-inverse. We define 
$$
\widehat{\mathcal{R}}(Y) = \wh{\alpha}
$$
and we claim that this is actually a good approximation, up to noise level, to $\mathcal R$ as soon as $n$ is large enough, without incurring in the curse of dimensionality (i.e., without exponential dependency of the computational complexity on $d$). More precisely, we prove that,
for a given $\tau>0$, with probability greater than $1 -6 \mathrm{e}^{-\tau^2}$, we have that 
\[
   \nor{Z^{\wh{\alpha}}-X }   \leq  \nor{Z^{\alpha^*}-X} + \frac{1}{\sigma_d}B(n, \tau,\sigma), 
\]
where $\sigma_d$ is the smallest positive singular value of $A$.
Let us stress that $B(n, \tau,\sigma)$ gets actually small for small $\sigma$ and for  $n =\mathcal O(m\times h)$ (see formula \eqref{eq:59}) and $\widehat{\mathcal{R}}(Y) = \wh{\alpha}$ is $\sigma$-optimal.
We  provide an explicit expression for $B$ in Proposition \ref{prop:boundB}. In the special case where $A= I$ we derive a bound on the difference between the learned parameter $\wh{\alpha}$ and the optimal parameter $\alpha^*$, see Theorem \ref{thm:main}, justifying even more precisely the approximation $\widehat{\mathcal{R}}(Y) = \wh{\alpha} \approx \alpha^* = \mathcal R(Y)$.


The paper is organized as follows: After introducing some notation and problem set-up in the next section, we provide the accuracy bounds on the learned estimators with respect to their distribution dependent counterparts in Section 3. For the special case $A = I$ we provide an explicit bound on the difference between the learned and the optimal regularization parameter and discuss the amount of samples needed for an accurate learning in Section 4. We also exemplify the presented theoretical results with a few numerical illustrations. Section 5  provides explicit formulas by means of numerical linearization for the parameter learning. 
Section 6 offers a snapshot of the main contributions and presents a list of open questions for future work.
Finally, Appendix A and Appendix B contain some background information
on perturbation theory for compact operators, the sub-gaussian random variables, and proofs of some technical theorems, which are valuable for understanding the scope of the paper.

\section{Setting}

This section presents some background material and sets the notation for the rest of the work.
First, we fix some notation.  The Euclidean norm of a vector $v$ is denoted
by $\nor{v}$ and the Euclidean scalar product between two vectors
$v,w$ by $\scal{v}{w}$.  We denote with $S^{d-1}$ the Euclidean unit sphere in $\mathbb R^d$. If $M$ is a matrix, $M^T$ denotes its
transpose, $M^\dagger$ the pseudo-inverse,
$M^{\dagger k}=(M^{\dagger})^k$ and $\nor{M}$ its spectral
norm. Furthermore, $\ker{M}$ and $\ran{M}$ are the null space and the
range of $M$ respectively. For a square-matrix  $M$, we use $\tr{M}$ to denote its trace. If
$v$ and $w$ are vectors (possibly of different length), $v\otimes w$
is the rank one matrix with entries $(v\otimes w)_{ij}=v_i w_j$.

Given a random vector $\xi\in\R^d$,  its noncentered covariance matrix is denoted by
\[ \Sigma_{\xi}=\mathbb E[ \xi\otimes \xi],\]
which is a positive matrix satisfying the following property
\begin{equation}
\ran{\Sigma_\xi}=(\ker{\Sigma_\xi})^\perp=\lspan{x\in\R^d\mid \mathbb P[ \xi\in
  B(x,r)]>0\ \forall r>0 }\label{eq:54},
\end{equation}
here $B(x,r)$ denotes the ball of radius $r$ with the center at $x$. A random vector $\xi$ is called sub-gaussian if
  \begin{equation}
    \label{eq:A2a}
    \nor{\xi}_{\psi_2}:= \sup_{v\in S^{d-1}}\sup_{q\geq 1} q^{-\frac{1}{2}}\mathbb
    E[ \abs{\scal{\xi}{v}}^q]^{\frac{1}{q}} <+\infty.
  \end{equation}
  The value $\nor{\xi}_{\psi_2}$ is called the sub-gaussian norm of
  $\xi$ and the space of sub-gaussian vectors becomes a normed vector space
  \citep{ver12}. Appendix~\ref{sec:sub-gaussian-vectors} reviews  some basic  properties about sub-gaussian vectors.

We consider the  following class of inverse problems.
\begin{assumption}\label{ass}
In the statistical linear inverse problem 
\[Y= A X + \sigma W,\] 
the following conditions hold true:
  \begin{enumerate}[a)]
  \item $A$ is an $m\times d$-matrix with norm  $\nor{A}=1$;
  \item the signal $X\in \R^d$ is a sub-gaussian random vector with
    $\nor{X}_{\psi_2}=1$; 
  \item\label{item:noise} the noise $W \in \R^m$ is a sub-gaussian centered random
    vector independent of $X$ with $\nor{W}_{\psi_2}=1/\sqrt{2}$ and
    with the noise level $0<\sigma<\sqrt{2}$; 
  \item\label{item:1} the covariance matrix $\Sigma_X$ of $X$ has a low
    rank matrix, {\em i.e.,}
    \[ \operatorname{rank}(\Sigma_X)=h\ll d. \]
  \end{enumerate}
\end{assumption}
We add some comments on the above conditions. 
The normalisation assumptions on $\nor{A}$, $\nor{X}_{\psi_2}$ and
$\nor{W}_{\psi_2}$ are stated only to simplify the 
bounds. They can always be satisfied by  rescaling $A$, $X$ and
$W$ and our results hold true by replacing $\sigma$ with 
$\sqrt{2} \nor{W}_{\psi_2}\nor{A}^{-1}\nor{X}^{-1}_{\psi_2}   \sigma$.  The
upper bound on $\sigma$ reflects the intuition that $\sigma W$ is a
small perturbation of the noiseless problem. 
 
Condition~\ref{item:1}) means that $X$ spans a low dimensional
subspace of $\R^d$. Indeed, by~\eqref{eq:54} condition~\ref{item:1}) is equivalent to the
fact that the vector space
\begin{equation}
\mathcal V=  \label{eq:21} \ran{\Sigma_X}=\lspan{ x\in\R^d\mid \mathbb P[X\in
B(x,r)]>0 \text{ for all } r>0 }
\end{equation} 
is an $h$-dimensional subspace and $h$ is the dimension
of the minimal subspace containing $X$ with probability 1, {\em i.e.,}
\[ 
h = \min _{\mathcal K}\dim{\mathcal K },
\] 
where the minimum is taken over all subspaces $\mathcal
K\subset\R^d$ such that $\mathbb P[X\in\mathcal K]=1$. 

We write $a\lesssim b$ if there exists an absolute  constant $C$ such
that $a\leq C b$. By {\em absolute} we mean that it  holds for all the
problems $Y=AX+\sigma W$ satisfying Assumption~\ref{ass}, in
particular, it is independent of $d,m$ and $h$.

The datum $Y$ depends only on the projection
$X^{\dagger}$ of $X$ onto $\ker{A}^\perp$ and $Z^\alpha$ as solutions of \eqref{eq:13} also belong to $\ker{A}^\perp$. Therefore,  we can always assume, without loss of generality, for the rest of the paper that $A$ is {\it injective} by
replacing $X$ with  $X^\dagger$, which is a sub-gaussian random
vector,  and $\R^d$  with $\ker{A}^\perp$. 

Since $A$ is injective, $\operatorname{rank}(A)=d$ and we define the
singular value decomposition of $A$ by $(u_i,v_i,\sigma_i)_{i=1}^d,$
so that $A=U D V^T$ or 
$$
A v_i = \sigma_i u_i, \quad i=1,\dots,d,
$$
where  $\sigma_1 \geq \sigma_2 \geq \dots \geq \sigma_d>0$. Since $\nor{A}=1$,
clearly $\sigma_1=1$. Furthermore, let $Q$ be the projection onto the
$\lspan{u_1,\ldots,u_d}$,  so that  $Q A=A$, and we have the decomposition
\begin{equation}\label{eq:projdecomp}
Q = A A^\dagger.
\end{equation}
Recalling~\eqref{eq:21}, since $A$ is now assumed injective and 
\[ \Sigma_{AX}   = \mathbb E[ AX\otimes AX] =A \Sigma_X A^T,\] 
then 
\begin{equation}
\mathcal
W= \ran{\Sigma_{AX}}=(\ker{\Sigma_{AX}})^\perp=A\mathcal V,
\label{eq:61}
\end{equation} 
and, by condition~\ref{item:1}) in Assumption \ref{ass}, we have as well $\dim{\mathcal W}=h$. 

We denote by $\Pi$ the projection onto $\mathcal W$ and by
\begin{equation}\label{eq:pdef}
p=\max\set{ i\in\set{1,\ldots,d} \mid \Pi u_i\neq 0},
\end{equation}
so that, with probability $1$, 
\begin{equation}\label{eq:15}
\Pi A X = AX \quad \mbox{ and } \quad  X= \sum_{i=1}^p \langle X, v_i \rangle v_i.
\end{equation}
Finally, the random vectors $\eta = \sigma W$, $AX,$ and $Y$ are
sub-gaussian and take value in $\R^m$, $\mathcal W,$ and $\R^m$,
respectively, with
\begin{equation}
  \label{eq:22}
\|  A X\|_{\psi_2} \leq \|A^T\| \| X\|_{\psi_2} = 1 \qquad  \|
 Y\|_{\psi_2} \leq \|A X\|_{\psi_2}+\sigma\nor{W}_{\psi_2} \leq 2
\end{equation}
since, by Assumption \ref{ass}, $\|A\|=1$ and $\sigma\leq \sqrt{2}$.

For any $t\in [0,1]$ we set $Z^t$ as the  solution
of the minimization problem  
\begin{equation}
  \label{eq:23}
   \min_{z\in\R^d} \left(t\, \nor{A z-Y}^2 + (1-t)\, \nor{z}^2\right),
\end{equation}
which is the solution of the  Tikhonov functional 
\[
\min_{z\in\R^d} \nor{Az-Y}^2 + \alpha\, \nor{z}^2. 
\]
with $\alpha=(1-t)/t\in [0,+\infty]$. 

For $t<1$, the solution is unique, for $t=1$ the minimizer  is not unique and we set
$$Z^1 = A^\dag Y.$$

The explicit form of the solution of  \eqref{eq:23} is given by
\begin{align}
Z^t & =  t (t A^T A + (1-t)  I)^{-1} A^T Y \label{eq:1}\\
      & =  \sum_{i=1}^d \frac{t \sigma_i}{ t \sigma_i^2 + (1-t)}
        \langle Y, u_i \rangle v_i  \nonumber\\
      & =  \sum_{i=1}^d \left( \frac{t \sigma_i^2}{ t \sigma_i^2 +
        (1-t)}	\langle X, v_i \rangle + \frac{t \sigma_i}{ t
        \sigma_i^2 + (1-t)} \langle \eta, u_i \rangle	 \right) v_i,\nonumber
\end{align}
which  shows that $Z^t$ is also a sub-gaussian random vector.

We seek for the optimal
parameter $t^* \in [0,1]$ that minimizes the
reconstruction error 
\begin{equation}\label{optparam}
\min_{t\in [0,1]} \nor{Z^t-X}^2.
\end{equation}
Since $X$ is not known,  the optimal parameter $t^*$ can not be
computed.  We assume that we have at disposal a training set of
$n$-independent noisy data
\[
Y_1,\ldots,Y_n,
\]
where $Y_i = A X_i+\sigma W_i$, and each pair $(X_i,W_i)$ is  distributed as $(X,W)$, for $i=1,\dots,n$.

We introduce the $\red{m\times m}$ empirical covariance matrix 
\begin{equation}
\wh{\Sigma}= \frac{1}{n} \sum_{i=1}^n Y_i\otimes Y_i,\label{eq:60}
\end{equation}
and we denote by $\wh{\Pi}$   the projections onto the
vector space spanned by the first $h$-eigenvectors of $\wh{\Sigma}$, where
the corresponding (repeated) eigenvalues are ordered in a nonincreasing way. 

\begin{remark}
The well-posedness of the empirical realization  $\wh{\Pi}$  in terms of spectral gap at the $h$-th eigenvalue will be given in Theorem~$\ref{thm:QQn}$, where we show that for $n$  large enough the $h+1$-th eigenvalue
 is strictly smaller than the $h$-th eigenvalue. 
\end{remark}

We define the empirical estimators of $X$ and $\eta$ as
\begin{equation}
  \label{eq:6}
  \widehat{X}= A^{\dagger} \wh{\Pi} Y \qquad\text{and} \qquad \widehat{\eta}=(Y- \wh{\Pi} Y),
\end{equation}

so that, by Equation~\eqref{eq:projdecomp},
\begin{equation}
  \label{eq:8}
  A \widehat{X} + Q \widehat{\eta}= Q Y.
\end{equation}

 Furthermore, we set
$\wh{t}\in [0,1]$ as the minimizer of 
\[
\min_{t\in [0,1]} \nor{Z^t-\widehat{X}}^2.
\]
If $\widehat{X}$ is close to $X$, we expect that the solution $Z^{\wh{t}}$ has a
reconstruction error close to the minimum value. In the following
sections, we study the statistical properties of $\wh{t}$. However, we
 first provide some a priori information on the
optimal regularization parameter $t^*$.

\subsection{Distribution dependent quantities}
We define the  function $t\mapsto
\nor{R(t)}^2$, where 
\[
R(t)= Z^t- X\qquad t\in [0,1]
\]
is  the reconstruction error vector. Clearly, the function $t\mapsto \nor{R(t)}^2$ is continuous,
so that a  global  minimizer $t^*$ always exists in the compact interval $[0,1]$. 

Define for all $t\in [0,1]$ the $d\times d$ matrix 
\begin{alignat*}{1}
  B(t) & =t A^TA+(1-t) I= \sum_{i=1}^d (t\sigma_i^2 + 1-t)\,
  v_i\otimes v_i,
\end{alignat*}
which is invertible since $A$ is injective and its  inverse is 
\[
B(t)^{-1} = \sum_{i=1}^d \frac{1}{t\sigma_i^2 + 1-t}\,
  v_i\otimes v_i .
\]
Furthermore, $B(t)$ and $B(t)^{-1}$ are smooth functions of the parameter $t$ and
\begin{equation}
B'(t)= (A^TA- I)\qquad  (B(t)^{-1})'= - B(t)^{-2} (A^TA-
I).\label{eq:49}
\end{equation}
Since $Y=AX+\eta$, expression~\eqref{eq:1} gives
\begin{alignat}{1}
  R(t) & = t B(t)^{-1} A^T Y -X\label{eq:2} \\
        & = t B(t)^{-1} A^T (AX+ \eta) - X \nonumber \\
       & =  B(t)^{-1} ( t A^TAX -  B(t) X + t A^T\eta) \nonumber \\
        & =  B(t)^{-1} ( -(1-t) X + t A^T\eta). \nonumber
\end{alignat}
Hence,
\begin{align*}
\| R(t) \|^2 & =\nor{ B(t)^{-1} ( -(1-t) X + t A^T\eta)}^2 \label{eq:28}\\
& = 
\sum_{i =1}^d \left(  \frac{-(1-t)\xi_i + t \sigma_i \nu_i }{t \sigma_i^2 + (1-t)}   \right)^2,\nonumber
\end{align*}
where for all $i=1,\ldots,d$
\[
\xi_i= \langle X, v_i \rangle \qquad \nu_i= \langle \eta, u_i \rangle.
\]
In order to
characterize $t^*$  we may want to seek it among the zeros of the following
function
\[
H (t) =  \frac12 \frac{d}{dt} \| Z^t - X \|^2= \scal{R(t)}{R'(t)}.
\]
Taking into account~\eqref{eq:49}, the differentiation of~\eqref{eq:2} is given by
\begin{alignat}{1}
  R'(t) & =  B(t)^{-1} A^T Y - t B(t)^{-2}(A^TA-I) A^TY \label{eq:3} \\
        & =  B(t)^{-2} ( B(t)- t (A^TA-I)) A^TY \nonumber\\
&  = B(t)^{-2}A^TY, \nonumber
\end{alignat}
so that
\begin{alignat}{1}
  \label{eq:4}
   H(t) & =  \scal{AB(t)^{-3}  ( -(1-t) X + t A^T\eta)  }{A X +  \eta } \\
          & =   \sum_{i=1}^d \sigma_i\frac{-(1-t)\xi_i + t \sigma_i \nu_i }{(t \sigma_i^2 + (1-t))^3} (\xi_i \sigma_i + \nu_i) \nonumber \\
	& = \sum_{i=1}^d  \sigma_i\xi_i (\xi_i \sigma_i+ \nu_i)
        \frac{(\sigma_i \nu_i \xi_i^{-1}+1)t -1 }{ (1 - (1-\sigma_i^2)
          t)^3} \nonumber \\
  & = \sum_{i=1}^d \sigma_i\alpha_i h_i (t), \nonumber
\end{alignat}
where $\alpha_i =\xi_i( \sigma_i \xi_i + \nu_i)$ and 
$h_i(t) = \frac{(\sigma_i \nu_i \xi_i^{-1}+1)t -1 }{ (1 - (1-\sigma_i^2) t)^3}.$ 

We observe that 
\begin{enumerate}[a)]
\item if $t=0$ ({\it i.e.,} $\alpha=+\infty$), $B(0)=I$, then
\[ H(0)= -\nor{AX}^2 \red{-} \scal{AX}{\eta},\] 
which is negative if $\nor{\Pi \eta}\leq \nor{AX}$, {\em i.e.,} for 
\begin{equation*}
    \sigma\leq \frac{ \nor{AX}}{\nor{\Pi W}}.
  \end{equation*}
 Furthermore, by construction,
\[  \mathbb E[H(0)]= -\tr{\Sigma_{AX}} < 0; \]
\item if $t=1$ ({\it i.e.,} $\alpha=0$), $B(\red{1})=A^TA$ and
  \begin{alignat*}{1}
    H(1) & =    \scal{A(A^TA)^{-3} A^T\eta  }{A X +  \eta }  \\
& = \nor{(AA^T)^{\dagger} \eta}^2 +\scal{(AA^T)^{\dagger} \eta}{ (A^T)^{\dagger} X},
  \end{alignat*}
which is positive if $\nor{(AA^T)^{\dagger} \eta}\geq \nor{(A^T)^{\dagger}
  X}$, for example, when 
\[ 
\sigma \geq \sigma_d   \frac{
  \nor{X}}{ \abs{\scal{W}{u_d}}}.\]
Furthermore, by construction,
\[  \mathbb E[H(1)]= \tr{\Sigma_{ (AA^T)^{\dagger}\eta }} >0.\]
\end{enumerate}

Hence, if the noise level satisfies 
\[
 \frac{\nor{X}}{ \abs{\scal{W}{u_d}}}\leq \sigma \leq  \frac{ \nor{AX}}{\nor{\Pi W}}
\]
the minimizer $t^*$ is in the open interval $(0,1)$ and it is a
zero of $H(t)$. If $\sigma$ is too small, there is no need
of regularization since we are dealing with a finite dimensional
problem. On the opposite side, if $\sigma$ is too big, the best
solution is the trivial one, {\em i.e.,} $Z^{t^*}=0$.

%
%

\subsection{Empirical quantities}

We replace $X$ and $\eta$ with their empirical counterparts defined
in~\eqref{eq:6}. By Equation~\eqref{eq:8} and reasoning as in  Equation \eqref{eq:2}, we obtain
\begin{alignat*}{1}
  \wh{R}(t)&= Z^t - \widehat{X} \\
& = t B(t)^{-1} A^TQY  - \widehat{X} \\
& =  B(t)^{-1} (-(1-t) \widehat{X}+ t A^T\widehat{\eta}), 
\end{alignat*}
and
\begin{align*}
\| \wh{R}(t) \|^2 & =\nor{ B(t)^{-1} ( -(1-t) \widehat{X}
                            + t A^T\widehat{\eta})}^2 \label{eq:28}\\ 
& =  \sum_{i =1}^d \left(  \frac{-(1-t)\widehat{\xi}_i + t \sigma_i
  \widehat{\nu}_i }{t \sigma_i^2 + (1-t)}   \right)^2,\nonumber 
\end{align*}
where for all $i=1,\ldots,d$
\[ 
\widehat{\xi}_i=\scal{ \widehat{X}}{v_i}\qquad\text{and} \qquad
\widehat{\nu}_i=\scal{\widehat{\eta}}{u_i} . 
\]
Clearly,
\begin{equation}
  \label{eq:5}
  \wh{R}'(t) = R'(t)=B(t)^{-2}A^TQY.
\end{equation}
From~\eqref{eq:4}, we get
\begin{alignat}{1}
 \label{eq:7}
   \wh{H}(t) &  = \scal{\wh{R}(t)}{\wh{R'}(t)}\\
&  = \scal{B(t)^{-3}  ( -(1-t) \widehat{X} + t A^T\widehat{\eta})
   }{A^TA \widehat{X} +  A^T\widehat{\eta} ) } \nonumber\\
 & =  \sum_{i=1}^d \frac{-(1-t) \widehat\xi_i +t\sigma_i \widehat \nu_i }{ (1 -
   (1-\sigma_i^2) t)^3}  (\widehat\xi_i \sigma_i^2 + \sigma_i
 \widehat\nu_i) \nonumber, \\
& = \sum_{i=1}^d \sigma_i\widehat{\alpha}_i \widehat{h}_i (t)\nonumber,
\end{alignat}
where $\widehat{\alpha}_i =\widehat{\xi}_i( \sigma_i \widehat{\xi}_i +
\widehat{\nu}_i)$  and 
$\widehat{h}_i(t) = \frac{(\sigma_i \widehat{\nu}_i \widehat{\xi}_i^{-1}+1)t -1 }{ (1 - (1-\sigma_i^2) t)^3}.$

An alternative form in terms of $Y$   and
$\wh{\Pi}$, which can be useful as a different numerical implementation, is 
\begin{alignat*}{1}
  \label{eq:29}
 \wh{H}(t)  & = \scal{B(t)^{-1} (t A^T(Y- \wh{\Pi} Y) -(1-t)
   \red{A^{\dagger}} \wh{\Pi} Y)}{B(t)^{-2}  A^TY } \\
& = \scal{ t AA^T(Y- \wh{\Pi} Y) -(1-t)
   Q\wh{\Pi} Y }{(t AA^T+(1-t)I)^{\dagger 3} Q Y  } \nonumber \\
& = \scal{ t AA^T(Y- \wh{\Pi} Y) -(1-t)
   \wh{\Pi} Y }{(t AA^T+(1-t)I)^{\dagger 3} Q Y  } . \nonumber
\end{alignat*}
As for $t^*$, the   minimizer $\wh{t}$ of the function $t\mapsto \nor{\wh{R}(t)}^2$ 
always exists in $[0,1]$ and, for $\sigma$ in the range of
  interest, it is in the open
interval $(0,1)$, so that it is a zero of the function $\wh{H}(t)$.

\section{Concentration inequalities}
In this section, we bound the difference between the empirical
estimators and their distribution dependent counterparts.

By~\eqref{eq:61} and item~\ref{item:1}) of Assumption~\ref{ass}, the
covariance matrix $\Sigma_{AX}$ has rank $h$ and,
we set $\la_{\min}$ to be the smallest non-zero eigenvalue of
$\Sigma_{AX}$.   Furthermore, we denote by  $\Pi^{Y}$ the projection from $\mathbb R^m$
onto the vector space  spanned by the eigenvectors of $\Sigma_Y$ 
\red{with corresponding eigenvalue greater than $\la_{\min}/2$}.

The following proposition shows that $\Pi^{Y}$ is close to $\Pi$ if
the noise level is small enough.
\begin{proposition}\label{thm:QQ}
If $\sigma^2<{\la_{\min}}/4$, then $\dim{\ran{\Pi^{Y}}}=h$ and
\begin{equation}
  \label{eq:24}
 \nor{\Pi^{Y}- \Pi} \le \frac{2 \sigma^2}{\la_{\min}} .
\end{equation}
\end{proposition}
\begin{proof}
Since $AX$ and $W$ are independent and $W$ has zero
mean, then
\begin{equation*}
  \label{eq:25}
  \Sigma_Y = \Sigma_{AX} + \sigma^2 \Sigma_W.
\end{equation*}
\red{ Since $\Sigma_W$ is a
positive matrix and $W$ is a sub-gaussian 
vector satisfying~\eqref{eq:22},
with the choice $q=2$ in~\eqref{eq:A2a}, we have
\begin{equation}
  \label{eq:26}
 \nor{\Sigma_W} =\sup_{v\in S^{m-1}}\scal{\Sigma_W v}{v}=\sup_{v\in S^{m-1}}\mathbb E[
   \scal{W}{v}^2]\leq 2 \nor{W}^2_{\psi_2} = 1,
\end{equation}
so that $\nor{\Sigma_Y-\Sigma_{AX}}\leq\sigma^2<\la_{\min}/4$. 

We now apply Proposition~\ref{propro} with $\mathcal
A=\Sigma_{AX}$ and eigenvalues
$(\alpha_j)_j$ and projections\footnote{In the
  statement of  Proposition~\ref{propro} the eigenvalues are counted
  without their multiplicity and ordered in a decreasing way and each $P_j$
  is the projection onto the vector space
spanned by the eigenvectors {with corresponding eigenvalue greater or
  equal than  $\alpha_j$.}} $(P_j)_j$,  and $\mathcal B=\Sigma_Y$
 with eigenvalues $(\beta_\ell)_\ell$ and projections $(Q_\ell)_\ell$. We choose $j$ such that $\alpha_j=\la_{\min}$ so
   that $\alpha_{j+1}=0$, $P_j=\Pi$ and, by~\eqref{eq:61},
\[ \dim{\ran{P_j}}=\dim{\ran{\Pi}}=\dim{\ran{\Sigma_{A X}}}=h.\]
Then there exists $\ell$ such that $\beta_{\ell+1}<\la_{\min}/2<\beta_\ell$, so that
$Q_\ell=\Pi^{Y}$ and it holds that
$\dim{\ran{\Pi^{Y}}}=\dim{\ran{P_j}}=h$. Finally, ~\eqref{eq:A1}
implies~\eqref{eq:24} since $\alpha_{h+1}=0$. 
}
\end{proof}

Recall that $\wh{\Pi}$ is  the projection onto the
vector space spanned by the first $h$-eigenvectors of $\wh{\Sigma}$
defined by~\eqref{eq:60}. 
\begin{theorem}\label{thm:QQn}
Given $\tau>0$ with
probability greater than $1 -2 \mathrm{e}^{-\tau^2}$,
$\wh{\Pi}$ coincides with the projection onto the vector space
spanned by the eigenvectors of $\wh{\Sigma}$ \red{with corresponding eigenvalue greater than  $\la_{\min}/2$}. Furthermore
\begin{equation}\label{eq:27}
\nor{\wh{\Pi}- \Pi} \lesssim \frac{1 }{\la_{\min}}\left(
 \sqrt{\frac{m}{n}}+  \frac{\tau}{\sqrt{n}} +\sigma^2\right),
\end{equation}
provided that
\begin{align}
  n  & \gtrsim (\sqrt{m}+\tau)^2 \max\left \{\frac{64}{\la_{\min}^2},1 \right \}\label{cond}\\
  \sigma^2 & < \frac{\la_{\min}}{8} .\nonumber
\end{align} 
\end{theorem}
\begin{proof}
\red{
We apply Theorem~\ref{thm:concen} with $\xi_i=Y_i$ and
\begin{equation}
\delta= C\sqrt{\frac{m}{n}}+  \frac{\tau}{\sqrt{
    n}}\leq\min\set{1,\la_{\min}/8}\leq 1,\label{eq:9}
\end{equation}
by~\eqref{eq:26}. Since $\delta^2\leq\delta$, with probability greater
than $1-2 \mathrm{e}^{-\tau^2},$ 
\begin{align*}
 \nor{\wh{\Sigma}- \Sigma_{AX}}& \leq \nor{\wh{\Sigma}-
                           \Sigma_Y}+\nor{\Sigma_Y-\Sigma_{AX}} \\
& \leq C\left(
      \sqrt{\frac{m}{n}}+
  \frac{\tau}{\sqrt{ n}}\right)+ \sigma^2\\
& \leq \frac{\la_{\min}}{8} + \frac{\la_{\min}}{8} =\frac{\la_{\min}}{4},
\end{align*}
where the last inequality follows by~\eqref{eq:9}. 

As in the proof of Proposition~\ref{thm:QQ}, we apply
Proposition~\ref{propro} with $\mathcal A=\Sigma_{AX}$,
$\mathcal B=\wh{\Sigma}$ and $\alpha_j=\la_{\min}$ to be the smallest
eigenvalue of $\Sigma_{AX}$, so that $P_j=\Pi$ and $\dim{\ran{P_j}}=h$. 
Then, there exists a unique  eigenvalue $\beta_\ell$ of $\wh{\Sigma}$  such that
$\beta_{\ell+1}<\la_{\min}/2<\beta_\ell$ and
$\dim{\ran{Q_{\ell}}}=\dim{\ran{P_j}}=h$.  Then $Q_\ell= \wh{\Pi}$ and \eqref{eq:A1} implies~\eqref{eq:27}.   Note that the constant $C$
depends on $\nor{Y}_{\psi_2}\le 2$ by \eqref{eq:22}, so that it
becomes an absolute constant, when considering the worst case $\nor{Y}_{\psi_2}= 2$.}
\end{proof}
If $n$ and $\sigma$ satisfy~\eqref{cond},  the above proposition shows that the empirical covariance matrix
  $\wh{\Sigma}$ has a spectral gap around the value $\la_{\min}/2$ and
  the number of eigenvector \red{with corresponding eigenvalue greater than
  $\la_{\min}/2$} is precisely $h$, so that $\wh{\Pi}$ is uniquely
  defined. Furthermore, the dimension $h$ can be
  estimated by observing spectral gaps in the singular value decomposition of $\wh{\Sigma}$.

If the number $n$ of samples goes to infinity,
  bound~\eqref{eq:27} does not converge to zero due to term
  proportional to the noise level $\sigma$. However, if the random
  noise $W$ is isotropic, we can improve the estimate.

\begin{theorem}\label{thm:QQnbis}
Assume that $\Sigma_W=\operatorname{Id}$. 
Given $\tau>0$ with
probability greater than $1 -2 \mathrm{e}^{-\tau^2}$,
\begin{equation}\label{eq:27bis}
\nor{\wh{\Pi}- \Pi} \lesssim \frac{1 }{\la_{\min}}\left(
 \sqrt{\frac{m}{n}}+  \frac{\tau}{\sqrt{n}} \right),
\end{equation}
provided that
\begin{align}
  n  & \gtrsim (\sqrt{m}+\tau)^2 \set{1,\frac{16}{\la^2_{\min}}}\label{cond_bis}\\
  \sigma^2 & < \frac{\la_{\min}}{2} .\nonumber
\end{align} 
\end{theorem}
\begin{proof}
As in the proof of Proposition~\ref{thm:QQ}, we have that
\[
 \Sigma_Y = \Sigma_{AX} + \sigma^2 \Sigma_W= \Sigma_{AX} + \sigma^2
 \operatorname{Id},
\]
where the last equality follows from the assumption that the noise is
isotropic. Hence, the matrices $\Sigma_Y$ and $\Sigma_{AX}$ have the same eigenvectors,
whereas the corresponding eigenvalues are shifted by
$\sigma^2$. Taking into account that $\la_{\min}$ is the smallest
non-zero eigenvalue of $\Sigma_{AX}$ and denoted  by
$(\alpha_j)_{j=1}^N$ the eigenvalues of $\Sigma_Y$, it follows that
there exists $j=1,\ldots,N$ such that
\[ \alpha_1> \alpha_2 >\alpha_j=\la_{\min}+\sigma^2\qquad
  \alpha_{j+1}=\ldots=\alpha_N=\sigma^2.\]
Furthermore, denoted by $P_j$ the projection onto the vector space
spanned by the eigenvectors with corresponding eigenvalue greater or
  equal than  $\alpha_j$, 
it holds that  $\Pi=P_j$. By assumption $\sigma^2 <
\la_{\min}/2$, so that $\Pi^Y=P_j=\Pi$ and, hence, $\dim{\ran{P_j}}=h$.

As in the proof of Theorem~\ref{thm:QQn}, with probability 
$1-2 \mathrm{e}^{-\tau^2},$ 
\begin{align*}
 \nor{\wh{\Sigma}- \Sigma_Y} & \leq C\left(
      \sqrt{\frac{m}{n}}+  \frac{\tau}{\sqrt{ n}}\right) <
                               \min\set{1,\frac{\la_{\min}}{4}}
<\frac{\alpha_h-\alpha_{h+1}}{4} ,
\end{align*}
where $n$ is large enough, see~\eqref{cond_bis}. Then, there exists a unique eigenvalue $\beta_\ell$ of $\wh{\Sigma}$  such that
$\beta_{\ell+1}< \frac{\la_{\min}}{2}+ \sigma^2<\beta_\ell$ and
$\dim{\ran{Q_{\ell}}}=\dim{\ran{P_j}}=h$.  Then $Q_\ell= \wh{\Pi}$ and \eqref{eq:A1} implies~\eqref{eq:27bis}. 
\end{proof}

We need the following technical lemma.
\begin{lemma}
Given $\tau>0,$ with probability greater than $1 -4
\mathrm{e}^{-\tau^2}$, simultaneously it holds
\begin{equation}
  \label{eq:55}
 \nor{X} \lesssim ( \sqrt{h} + \tau ) \qquad \nor{Y} \lesssim  (
 \sqrt{h} +  \sigma   \sqrt{m} + \tau ) \qquad \nor{\Pi W} \lesssim ( \sqrt{h} + \tau ).
\end{equation}
\end{lemma}
\begin{proof}
Since $X$ is a sub-gaussian random vector taking values in $\mathcal
V$ with $h=\dim{\mathcal V}$, taking into account that
$\nor{X}_{\psi_2} =1$, bound~\eqref{eq:56} gives
\[ \nor{X} \leq 9 (\sqrt{h} + \tau),\]
 with probability greater than $1-2 {\mathrm e}^{-\tau^2}$. Since $W$ is a
   centered sub-gaussian random vector taking values in $\R^m$ and
$\nor{W}_{\psi_2}\leq 1$, by~\eqref{eq:57} 
\[ \nor{W} \leq 16 (\sqrt{m} + \tau),\]
with probability greater than $1- {\mathrm e}^{-\tau^2}$.  Since $\nor{A}=1$ and
\[\nor{Y}\leq \nor{AX}+\sigma \nor{W} \leq \nor{X}+\sigma \nor{W}, \]
the first two bounds in~\eqref{eq:55} hold true with
probability greater than $1-3 {\mathrm e}^{-\tau^2}$. Since $\Pi W$ is a
   centered sub-gaussian random vector taking values in $\mathcal W$
   with $h=\dim{\mathcal W},$ and $\nor{\Pi W}_{\psi_2}\leq 1$,  by~\eqref{eq:57} 
\[ \nor{\Pi W} \leq 16 (\sqrt{h} + \tau),\]
with probability greater than $1- {\mathrm e}^{-\tau^2}$. 
\end{proof}
As a consequence, we have the following bound. 
\begin{proposition}
\label{prop:boundB}
 Given $\tau>0$, if  $n$ and $\sigma$ satisfy~\eqref{cond},
then  with probability greater than $1 -6
\mathrm{e}^{-\tau^2}$
\begin{equation}
  \label{eq:58}
  \nor{ (\Pi - \wh{\Pi}) Y - \Pi\eta}\lesssim  B(n,\tau,\sigma),
\end{equation}
where
\begin{alignat}{1}
  \label{eq:59}
  B(n,\tau,\sigma) & = \frac{1}{\la_{\min}}
\sqrt{\frac{hm}{n}}+\sigma
\left(
\sqrt{h}+ \frac{1}{\la_{\min}}\frac{m}{\sqrt{n}}\right)+ \frac{\sigma^2}{ \la_{\min}} \sqrt{h} +
\frac{\sigma^3}{ \la_{\min}} \sqrt{m} +\\
& \quad + \tau \left( \frac{1}{\la_{\min}} \sqrt{\frac{m}{n}}
  + \sigma  \left(1+\frac{1}{\la_{\min}}
    \sqrt{\frac{m}{n}}\right) +\frac{\sigma^2}{ \la_{\min}}
\right) +\tau^2 \frac{1}{\la_{\min}} \frac{1}{\sqrt{n}} .\nonumber
\end{alignat}
\end{proposition}
\begin{proof}
Clearly,
\[
\nor{(\Pi - \wh{\Pi}) Y - \Pi\eta}\leq \nor{\Pi-\wh{\Pi}} \nor{Y} + \sigma \nor{\Pi W}.
\] 
If \eqref{cond} holds true,  bounds~\eqref{eq:27} and~\eqref{eq:55}  imply 
\[ \nor{(\Pi - \wh{\Pi}) Y - \Pi\eta}\lesssim\frac{1 }{\la_{\min}}\left(\sqrt{\frac{m}{n}}+
      \frac{\tau}{\sqrt{n}} +\sigma^2\right) ( \sqrt{h} + \sigma
    \sqrt{m} + \tau ) + \sigma ( \sqrt{ h} + \tau ), \]
with probability greater than $1 -6 \mathrm{e}^{-\tau^2}$.  By
developing the brackets and taking into account that $\sqrt{h+m}\leq
\sqrt{2m}$,~\eqref{eq:58} holds true. 
\end{proof}
\begin{remark}
  Usually in machine
learning bounds of the type~\eqref{eq:58} are considered in terms of their expectation, e.g.,   with respect to $(X,Y)$. In our framework, this would amount 
to the following bound
\begin{alignat*}{1}
  & \mathbb E\left[ \nor{(\Pi - \wh{\Pi}) Y - \Pi\eta}\,\Big|\, Y_1,\ldots,Y_n\right] \lesssim\\
  & \lesssim \quad\frac{1 }{\la_{\min}}\left(\sqrt{\frac{m}{n}}+
      \frac{\tau}{\sqrt{n}} +\sigma^2\right) (\sqrt{h}+\sigma
    \sqrt{m}) + \sigma \sqrt{h},
\end{alignat*}
obtained by observing that $\mathbb E[ \nor{Y}] \leq \mathbb E[ \nor{A} \nor{X}] + \sigma \mathbb
E[ \nor{W}]$, 
\[ \mathbb E[ \nor{X}]^2 \leq  \mathbb E[
\nor{X}^2]=\tr{\Sigma_X}\leq 2 h \nor{X}_{\psi_2}^2\lesssim h,\] 
and, by a similar computation, 
\begin{alignat*}{1}
\mathbb E[ \nor{W}]& \lesssim \sqrt{m}\qquad    \mathbb E[ \nor{\Pi W}]\lesssim \sqrt{h}.
\end{alignat*}
Our bound~\eqref{eq:58} is  much stronger and it holds in probability with respect to both the training
set $Y_1,\ldots Y_n$ and the new pair $(X,Y)$. 
\end{remark}
Our first result is a direct consequence of the  estimate~\eqref{eq:58}.
\begin{theorem}\label{thm:1}
Given $\tau>0$, with probability greater than $1 - 6 \mathrm{e}^{-\tau^2}$,
 \begin{alignat*}{1}
\nor{\widehat{X}-X} & \lesssim  \frac{1}{\sigma_d} B(n,\tau,\sigma) \\
\nor{Q \widehat{\eta}-Q \eta} & \lesssim  B(n,\tau,\sigma) \\
   \nor{Z^{\wh{t}}-X } - \nor{Z^{t^*}-X} & \lesssim  \frac{1}{\sigma_d}B(n, \tau,\sigma) \\
 \sup_{0\leq t\leq 1} \abs{\nor{\wh{R}(t)}-\nor{R(t)}}  & \lesssim  \frac{1}{\sigma_d} B(n, \tau,\sigma)
 \end{alignat*}
provided that $n$ and $\sigma$ satisfy~\eqref{cond}.
\end{theorem}
\begin{proof}
By the first identity of \eqref{eq:15}
  \begin{alignat}{1}
    \label{eq:47}
    X-\widehat{X} & = A^{\dagger}\Pi  A X -  A^{\dagger} \wh{\Pi} (AX+\eta) \\
    & = A^{\dagger} (\Pi - \wh{\Pi}) AX + A^{\dagger} (\Pi - \wh{\Pi})
    \eta -  A^{\dagger}\Pi\eta \nonumber\\
    & = A^{\dagger}\left( (\Pi - \wh{\Pi}) Y - \Pi\eta \right),\nonumber
  \end{alignat}
so that
\[
\nor{X-\widehat{X}}\leq \frac{1}{\sigma_d} \nor{(\Pi - \wh{\Pi}) Y - \Pi\eta}.
\]
An application of \eqref{eq:58} to the previous estimate gives the first bound of the statement. Similarly, we derive the second bound as follows.
Equations~\eqref{eq:8}, \eqref{eq:projdecomp}, and~\eqref{eq:47} yield 
  \begin{alignat*}{1}
    Q \eta -Q \widehat{\eta} & = A( X-\widehat{X} ) \\ 
    & = Q\left((\Pi - \wh{\Pi}) Y - \Pi \eta\right) .
  \end{alignat*}
  The other bounds follow by estimating them by multiples of $\nor{X-\widehat{X}}$ as we show below.
By definition of $\wh{t}$
\begin{alignat*}{1}
   \nor{Z^{\wh{t}}-X }& \leq   
     \nor{Z^{\wh{t}}-\widehat{X}} + \nor{X-\widehat{X}} \\
     &\leq \nor{Z^{t^*}-\widehat{X}}+ \nor{X-\widehat{X}}, \\
     & \leq \nor{Z^{t^*}-X}+ 2 \nor{X-\widehat{X}}.
\end{alignat*} 
Furthermore, 
  \begin{alignat}{1}
    \wh{R}(t)- R(t) & = X-\widehat{X} = A^{\dagger} \left((\Pi-\wh{\Pi})
      Y- \Pi\eta\right) \label{eq:10} ,
  \end{alignat}
and triangle inequality gives
\begin{equation*}
\sup_{0\leq t\leq 1} \abs{\nor{\wh{R}(t)}-\nor{R(t)}}\leq\sup_{0\leq
  t\leq 1} \nor{\wh{R}(t)- R(t) }= \nor{X-\widehat{X}}. \label{eq:48}
\end{equation*}
All the remaining bounds in the statement of the theorem are now consequences of~\eqref{eq:58}.
\end{proof}
\begin{remark}\label{rem:1} To justify and explain the consistency of the sampling strategy for approximation of the optimal regularization parameter $t^*$, let us assume
that $n$   goes to infinity and $\sigma$ vanishes. 
Under this theoretical assumption, Theorem~\ref{thm:1} shows that $\nor{\wh{R}(t)}$
  convergences uniformly to $\nor{R (t)}$ with high probability. The
  uniform convergence implies the $\Gamma$-convergence \citep[see][]{braides}, and, since 
  the domain $[0,1]$ is compact, Theorem~$1.22$ in \citet{braides} ensures that
 \[
\lim_{\underset{\sigma \to 0}{n\to+\infty}} \left(\inf_{0\leq t\leq 1} \nor{\wh{R}(t)} - \inf_{0\leq t\leq 1} \nor{R(t)} \right)=0.
\]
While the compactness given by the $\Gamma$-convergence guarantees the consistency of the approximation to an optimal parameter, it is much harder for arbitrary $A$ to provide an error bound, depending on $n$. For the case $A=I$ in Section $\ref{sec:AisI}$ we are able to establish very precise quantitative bounds  with high probability.
\end{remark}
\begin{remark}
Under the conditions of Theorem~$\ref{thm:1}$ for all $i=1,\ldots,d$ it
holds as well
\begin{alignat*}{1}
  \abs{\xi_i-\widehat{\xi_i}}&\lesssim \frac{1}{\sigma_i} B(n,\tau,\sigma)\\
  \abs{\nu_i-\widehat{\nu_i}}&\lesssim B(n,\tau,\sigma). 
\end{alignat*}
These bounds are a direct consequence of Theorem~$\ref{thm:1}$.
\end{remark}
The following theorem is about the uniform  approximation to the derivative function $H(t)$.
\begin{theorem}\label{thm:unifapprox}
Given $\tau>0$, with probability greater than $1 -10 \mathrm{e}^{-\tau^2}$,
 \begin{alignat*}{1}
  \sup_{0\leq t\leq 1}   \abs{\wh{H}(t) - H(t)}\lesssim B(n,\tau,\sigma)
\left( \frac{1}{\sigma_p^3} ( \sqrt{h}+ \tau)    + \frac{\sigma}{\sigma_d^4} ( \sqrt{d}+ \tau) \right)
 \end{alignat*}
provided that $n$ and $\sigma$ satisfy~\eqref{cond}, where $p$ is defined in \eqref{eq:pdef}.
\end{theorem}
\begin{proof}
Equations~\eqref{eq:5} and~\eqref{eq:10} give
\begin{alignat*}{1}
  \wh{H}(t) - H(t) & = \scal{\wh{R}(t)- R(t)}{R'(t)} \\ 
  &   = \scal{(\Pi-\wh{\Pi}) Y - \Pi\eta }{\red{(A^T)^{\dagger}}
    B(t)^{-2} A^TY}           \nonumber  \\
&   = \scal{(\Pi-\wh{\Pi})Y- \Pi\eta }{
    (t AA^T+(1-t)I)^{-2} QY} ,\nonumber 
\end{alignat*}
where we observe that $t AA^T+(1-t)I$ is invertible on $\ran{Q}$ \red{and $A^TY=A^TQY$}.
Hence,
\begin{alignat*}{1}
  \abs{\wh{H}(t) - H(t)} & \leq \nor{ (\Pi-\wh{\Pi})Y-
    \Pi\eta} \times\\
& \quad \times\left(\nor{ (t AA^T+(1-t)I)^{-2}AX} + \sigma \nor{(t AA^T+(1-t)I)^{-2}QW} \right).
\end{alignat*}
Furthermore, recalling that $AX=\Pi AX$ and $\Pi u_i=0$ for all
$i>p$, \eqref{eq:55} implies that
\begin{alignat*}{1}
  \nor{ (t
    AA^T+(1-t)I)^{-2}AX}  & \leq \frac{\sigma_p}{
    (t\sigma_p^2+(1-t))^2} \nor{X} \lesssim
  \frac{1}{\sigma_p^3} ( \sqrt{h}+ \tau) \\
\nor{(t AA^T+(1-t)I)^{-2}QW} & \leq \frac{1}{
    (t\sigma_{d}^2+(1-t))^2} \nor{QW}  \lesssim
  \frac{1}{\sigma_{d}^4} ( \sqrt{d}+ \tau) 
\end{alignat*}
hold  with probability greater than $1-4 {\mathrm e}^{-\tau^2}$.  
Bound~\eqref{eq:58} provides the desired  claim.
\end{proof}

The uniform approximation result of Theorem \ref{thm:unifapprox}
allows us to claim that any $\wh t \in [0,1]$ such that $\wh
H(\wh t)=0$ can be attempted as a proxy for the optimal
parameter $t^*$, especially if it is the only root in the interval
 $(0,1)$. 
 
Nevertheless, being $\widehat H_{n}$ a sum of $d$  rational
functions of polynomial numerator of degree $1$ and polynomial
denominator of degree $3$, the computation of its zeros in $[0,1]$ is
equivalent to the computation of the roots of a polynomial of degree
$3(d-1)+1=3 d-2$. The computation cannot be done analytically for
$d>2,$ because it would require the solution of a polynomial equation
of degree larger than $4$. For $d>2,$ we are forced to use numerical
methods, but this is not a great deal as by now there are plenty of
stable and reliable routines to perform such a task (for instance, Newton method,
numerical computation of the eigenvalues of the companion matrix,
just to mention a few).  

We provide below relatively simple numerical experiments to validate the theoretical results reported above. In Figure \ref{fig2} we show optimal parameters $t^*$ and corresponding approximations $\wh t$ (computed by numerical solution to the scalar nonlinear equation $\widehat H_n(t) =0$ on $[0,1]$), for $n$ different data $Y =A X +\eta$. The accordance of the two parameters $t^*$ and $\wh t$ is visually very convincing and their statistical (empirical) distributions reported in Figure \ref{fig2stat} are also very close.

 \begin{figure}[h!]
  \centering
    \includegraphics[width=0.6\textwidth]{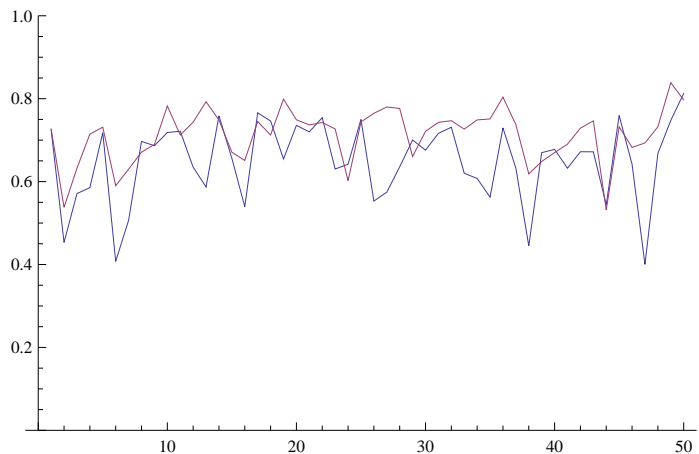}
  \caption{Optimal parameters $t^*$ and corresponding approximations $\wh t$ {\color{black} ($n=1000$)} for $50$ different new data $Y =A X +\eta$ for $X \in \mathbb R^d$ and $\eta\in \mathbb R^m$ Gaussian vectors, $d=200$, $m=60$ and $A \in \mathbb{R}^{60 \times 200}$. 
 Here we assumed that $X \in \mathcal V$ for $\mathcal V= \rm{span}\{e_1,\dots,e_5\}$.
  We designed the matrix in such a way that the spectrum is vanishing, {\emph i.e.,} $\sigma_{\min} \approx 0$. Here we considered as noise level $\sigma=0.03$, so that the optimal parameter $t^*$ is rather concentrated around $0.7$. The accordance of the two parameters $t^*$ and $\wh t$ is visually very convincing.}
\label{fig2}
\end{figure}

\begin{figure}[h!]
  \centering
    \includegraphics[width=0.4\textwidth]{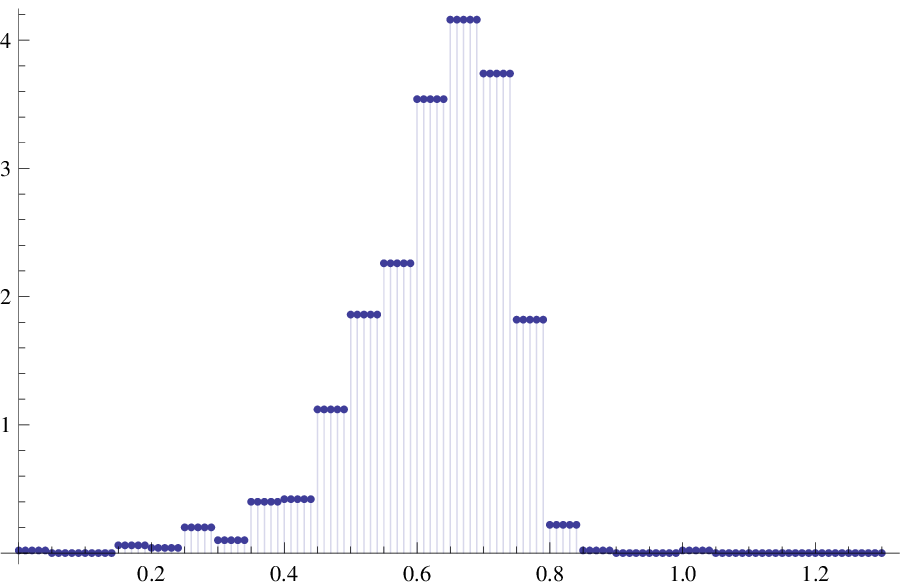}  \includegraphics[width=0.4\textwidth]{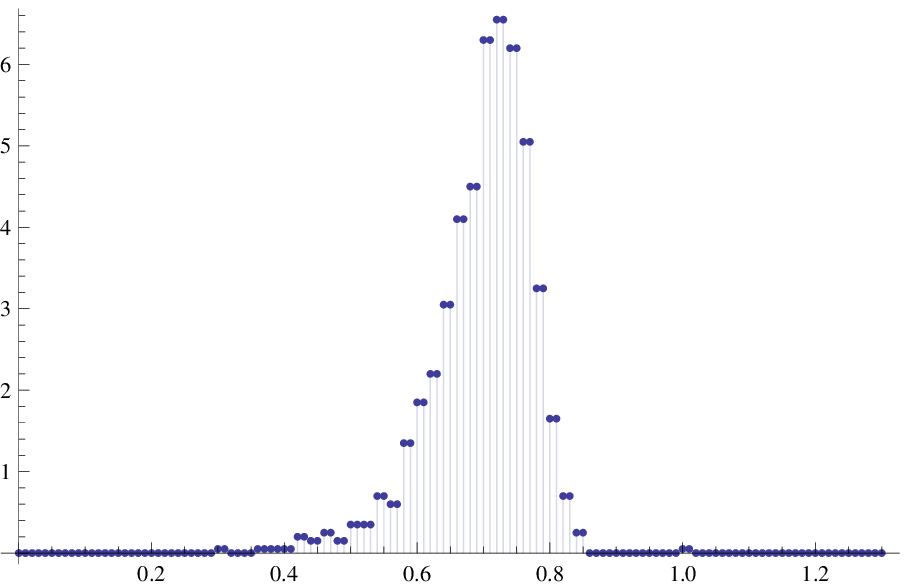}
  \caption{Empirical distribution of the optimal parameters $t^*$ (left) and the corresponding empirical distribution of the approximating parameters $\wh t$ (right) for $1000$ randomly generated
data $Y = A X +\eta$ with the same noise level. The statistical accordance of the two parameters $t^*$ and $\wh t$ is shown.}
\label{fig2stat}
\end{figure}

In the next two sections we discuss special cases where we can provide even more precise statements and explicit bounds.

\blue{
\section{Toy examples}\label{sec:examples}

In the following we specify the results of the previous in simple cases, which allow to understand more precisely the behavior of the different estimators presented in Theorem \ref{thm:1}.

\subsection{Deterministic sparse signal and Bernoulli noise in two dimensions}

From Theorem \ref{thm:1} one might conclude that the estimators $\widehat X, Z^{t^*}, Z^{\widehat t_n}$ are all performing an approximation to $X$ with a least error  proportional to $B(n,\tau,\sigma)$. With this in mind, one might be induced to conclude that computing $\widehat X$ is a sufficient effort, with no need of considering the empirical estimator $Z^{\widehat t_n}$, hence no need for computing $\widehat t_n$. In this section we discuss precisely this issue in some detail on a concrete toy example, for which explicit computations are easily performed.
\\

Let us consider a deterministic sparse vector $X=(1,0)^T \in \mathbb R^{2}$ and a stochastic noise $W=(W_1,W_2)^T\in \mathbb R^{2}$ with Bernoulli entries, i.e., $W_i=\pm 1$ with probability $1/2$. This noise distribution is indeed Subgaussian with zero mean.
For the sake of simplicity, we fix now an orthogonal matrix $A \in \mathbb R^{2\times 2}$, which we express in terms of its vector-columns as follows:
$
A = (A_1 | A_2)$.
Notice that $\|W\| = \sqrt 2$ with probability $1$ and $\|X\|=1$. Hence, we may want to consider a noise level $\sigma \in [-1/\sqrt 2, 1/\sqrt 2]$. 
In view of the fact that the sign of $\sigma$ would simply produce a symmetrization of the signs of the component of the noise, we consider from now on only the cases where $W_1\geq 0$. The other cases are obtained simply by changing the sign of $\sigma$.
\\
We are given measurements
$$
Y = A X + \sigma W.
$$
While $X$ is deterministic (and this is just to simplify the computations) $Y$ is stochastic given the randomness of the noise.
Hence, noticing that $Y = A X + \sigma W = A_1 + \sigma W$ and observing that $\mathbb E W= 0$, we obtain
$$
\Sigma_Y = \int Y Y^T d \mathbb P(W)= \int (A_1+ \sigma W)(A_1+ \sigma W)^T d  \mathbb P(W) = A_1 A_1^T + \sigma^2 I.
$$
It is also convenient to write $\Sigma_Y$ by means of its complete singular value decomposition
$$
\Sigma_Y = (A_1\, A_2) \left ( 
\begin{array}{ll}
1+ \sigma^2&0\\
0&\sigma^2
\end{array}
\right )
\left ( 
\begin{array}{l}
A_1^T\\
A_2^T
\end{array}
\right ).
$$

\subsection{Exact projection}

The projection onto the first largest singular space of $\Sigma_Y$ is simply given by
$$
\Pi \xi = \langle \xi ,A_1\rangle A_1.
$$
This implies that
$$
\Pi Y = A_1 + \sigma \langle W,A_1 \rangle A_1.
$$
From now on, we denote 
$$
\sigma_i = \sigma \langle W,A_i \rangle, \quad i=1,2.
$$
Let us stress that these parameters indicate how strong the particular {\it noise instance} $W$ is correlated with $\mathcal W = \operatorname{span} \{A_1 = A X\}$. If $|\sigma_1|$ is large, then there is strong concentration of the particular noise instance on $\mathcal W$ and necessarily $|\sigma_2|$ is relatively small. If, vice versa, $|\sigma_2|$ is large, then there is poor correlation of the particular noise instance with $\mathcal W$.\\
We observe now that, being $A$ orthogonal, it is invertible and
$$
I =A^T A = A^{-1} A = (A^{-1} A_1 \, A^{-1} A_2).
$$
This means that $A^{-1} A_1= (1,0)^T$ and $A^{-1} A_2= (0,1)^T$ and we compute explicitly
$$
\bar X = A^{-1} \Pi Y = A^{-1} A_1 + \sigma_1 A^{-1} A_1 = (1 + \sigma_1,0)^T.
$$
It is important to notice that for $\sigma_1$ small $\bar X$ is indeed a good proxy for $X$ and it does not belong to the set of possible
Tikhonov regularizers, i.e., the minimizers of the functional
\begin{equation}\label{tikreg}
 t \|A Z- Y \|^2_2+(1-t) \|Z\|^2_2
\end{equation}
for some $t \in [0,1]$. 
This proxy $\bar X$ for $X$ approximates it with error exactly computed by
\begin{equation}\label{errhat}
\| \bar X - X \| = |\sigma_1| \leq \sqrt 2 |\sigma|.
\end{equation}
The minimizer of \eqref{tikreg} is given by
\begin{eqnarray*}
Z^t &=& t ( tA^T A +(1-t) I)^{-1}A^T Y\\
&=& t ( t(A^T-A^{-1}) A + I)^{-1} \left ( 
\begin{array}{l}
A_1^T\\
A_2^T
\end{array}
\right ) (A_1 + \sigma W). 
\end{eqnarray*}
In view of the orthogonality of $A$ we have $A^T = A^{-1}$, whence  $ t(A^T-A^{-1}) A=0$, and the simplification
\begin{eqnarray*}
Z^t &=& t  \left ( 
\begin{array}{l}
A_1^T\\
A_2^T
\end{array}
\right ) (A_1 + \sigma W)\\
&=& t (1 + \sigma_1, \sigma_2)^T
\end{eqnarray*}
Now it is not difficult to compute
\begin{eqnarray*}
R(t)^2 = \| Z^t - X \|^2_2 =  (t (1+ \sigma_1) - 1)^2 + t^2 \sigma_2^2 = ((1+\sigma_1)^2+ \sigma_2^2) t^2 -2 t  (1+\sigma_1) + 1
\end{eqnarray*}
This quantity is optimized for
$$
t^* = \frac{1+\sigma_1}{(1+\sigma_1)^2+ \sigma_2^2}
$$
The direct substitution gives
\begin{eqnarray*}
R(t^*) &=& 
\frac{|\sigma_2|}{\sqrt{(1+\sigma_1)^2+ \sigma_2^2}}
\end{eqnarray*}
And now we notice that, according to \eqref{errhat} one can have either
$$
\| \bar X - X \| = |\sigma_1| < \frac{|\sigma_2|}{\sqrt{(1+\sigma_1)^2+ \sigma_2^2}} = \| Z^{t^*} - X \|
$$
or
$$
\| \bar X - X \| = |\sigma_1| > \frac{|\sigma_2|}{\sqrt{(1+\sigma_1)^2+ \sigma_2^2}} = \| Z^{t^*} - X \|
$$
very much depending on $\sigma_i = \sigma \langle W,A_i \rangle$. Hence, for the fact that $\bar X$ does not belong to the set of minimizers of \eqref{tikreg}, it can perfectly happen that it is a better proxy of $X$ than $Z^{t^*}$.
\\

Recalling that $\bar X = (1 + \sigma_1,0)^T$ and $Z^t = t (1 + \sigma_1, \sigma_2)^T$, let us now consider the error
\begin{eqnarray*}
\bar R(t)^2 &=& \| Z^t - \bar X\|^2\\
&=& 
((1+ \sigma_1)^2 + \sigma_2^2)t^2  -2 t (1+ \sigma_1)^2+ (1+ \sigma_1)^2.
\end{eqnarray*}
This is now opimized for 
$$
\bar t = \frac{(1+\sigma_1)^2}{(1+\sigma_1)^2+ \sigma_2^2}
$$
Hence we have
$$
R(\bar t)^2 =  \| Z^{\bar t} - X\|^2 = ((1+\sigma_1)^2+ \sigma_2^2) \bar t^2 -2 \bar t  (1+\sigma_1) + 1 = \frac{\sigma_1^2(1+\sigma_1)^2+ \sigma_2^2}{(1+\sigma_1)^2+ \sigma_2^2}
$$
or
$$
R(\bar t) = \frac{\sqrt{\sigma_1^2(1+\sigma_1)^2+ \sigma_2^2}}{\sqrt{(1+\sigma_1)^2+ \sigma_2^2}}
$$
In this case we notice  that, in view of the fact that $\sigma_1^2 \leq 1$ for $|\sigma| \leq 1/\sqrt{2}$ one can have  only
$$
\| \bar X - X \| = |\sigma_1| \leq \frac{\sqrt{\sigma_1^2(1+\sigma_1)^2+ \sigma_2^2}}{\sqrt{(1+\sigma_1)^2+ \sigma_2^2}}=|\sigma_1|  \frac{\sqrt{(1+\sigma_1)^2+ \sigma_2^2/\sigma_1^2}}{\sqrt{(1+\sigma_1)^2+ \sigma_2^2}}= \| Z^{\bar t} - X \|
$$
independently of $\sigma_i = \sigma \langle W,A_i \rangle$. The only way of reversing the inequality is by allowing noise level $\sigma$ such that
 $\sigma_1^2 > 1$, which would mean that we have noise significantly larger than the signal.

\subsubsection{Concrete examples}

Let us make a few concrete examples. 
Let us recall our assumpution that $\operatorname{sign}{W_1} =+1$ while  the one of $\operatorname{W_2}$ is kept arbitrary.
Suppose that $A_1 = \frac{W}{\|W\|}$, then $\sigma_1 = \sigma \sqrt 2$ and $\sigma_2=0$. In this case we get 
$$
\sigma \sqrt 2 = \| \bar X - X \| = \| Z^{\bar t} - X \|,
$$
and $\bar X = Z^{\bar t} $.
The other extreme case is given by $ A_2 = \frac{W}{\|W\|}$, then $\sigma_1 = 0$ and $\sigma_2=\pm \sigma \sqrt 2$.
In this case $\bar X=X$, while $R(\bar t)=R(t^*)=|\sigma| \sqrt{2/(1+2\sigma^2)}>0$. An intermediate case is given precisely by $A=I$, which we shall investigate in more generality in Section \ref{sec:AisI}. For this choice we have $\sigma_1 = \sigma$ and $\sigma_2=\pm \sigma$. Notice that the sign of $\sigma_2$ does not matter in the expressions of $R(\bar t)$ and $R(t^*)$ because it appears always squared. Instead the sign of $\sigma_1$ matters and depends on the choice of the sign of $\sigma$. We obtain
$$
R(t^*)= \frac{|\sigma|}{\sqrt{(1+\sigma)^2+ \sigma^2}},
$$
and
$$
R(\bar t)=\sqrt{\frac{\sigma^2(2+(\sigma(2+\sigma))}{1+2 \sigma(1+\sigma)}},
$$
and, of course
$$
\| \bar X - X \| = |\sigma|.
$$
In this case, $R(t^*)\leq \| \bar X - X \|$ if and only if $\sigma \geq 0$, while - unfortunately - $R(\bar t) > \| \bar X - X \|$ for all $|\sigma|<1$ and the inequality gets reversed only if $|\sigma|\geq1$.

 \begin{figure}[h!]
  \centering
    \includegraphics[width=1\textwidth]{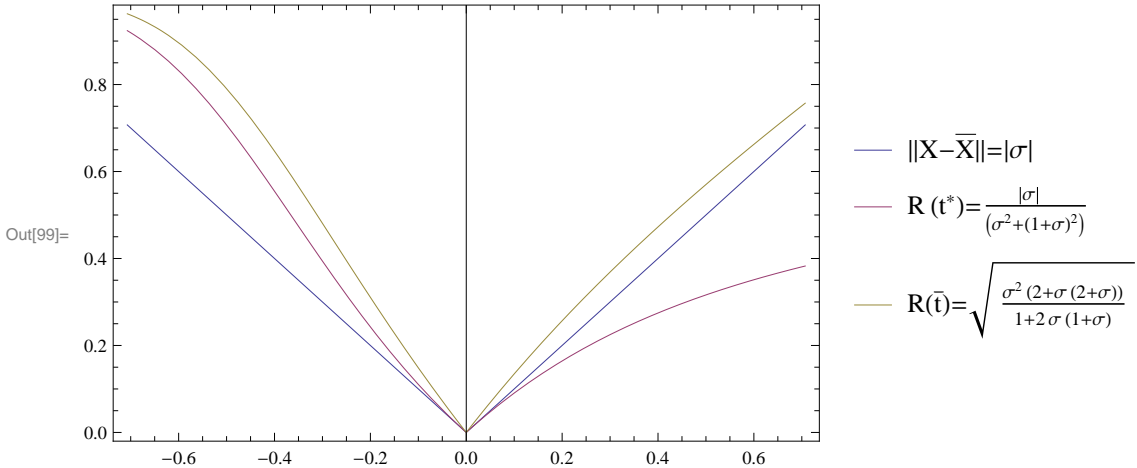}
  \caption{Comparison of $\| \bar X - X \|$, $R(t^*)$, and  $R(\bar t)$ as a function of $\sigma$ for $\sigma_1=\sigma$. }
\label{fig0:VE}
\end{figure}
In Figure \ref{fig1:VE} we illustrate the case of $\sigma_1 = 1.2 \sigma$
\begin{figure}[h!]
  \centering
    \includegraphics[width=1\textwidth]{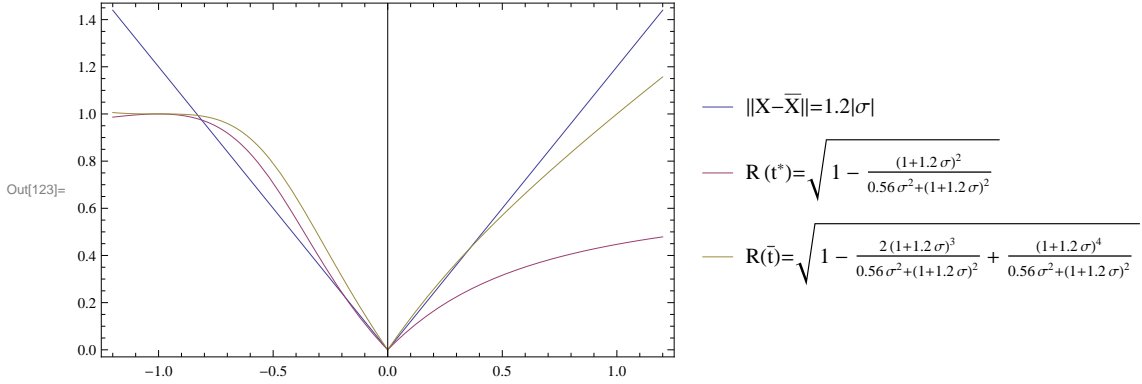}
  \caption{ Comparison of $\| \bar X - X \|$, $R(t^*)$, and  $R(\bar t)$ as a function of $\sigma$ for $\sigma_1=1.2 \sigma$. }
\label{fig1:VE}
\end{figure}
and in Figure \ref{fig2:VE} the case of $\sigma_1 = 0.7 \sigma$.
\begin{figure}[h!]
  \centering
    \includegraphics[width=1\textwidth]{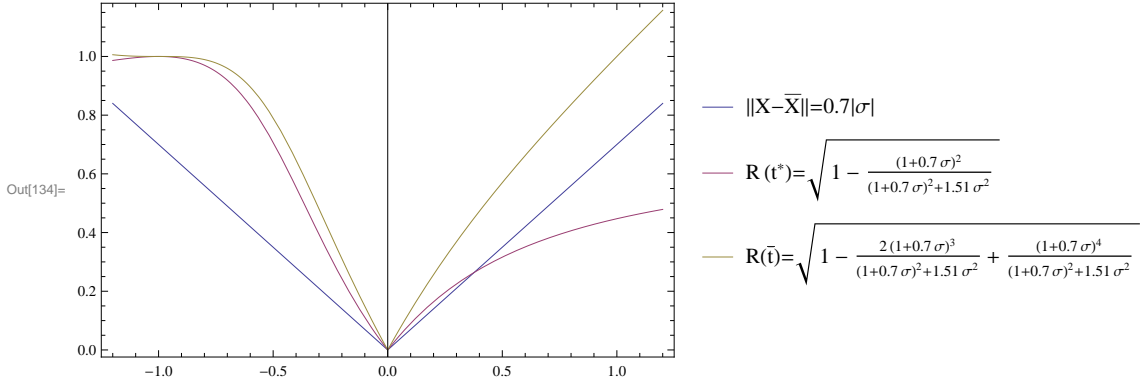}
  \caption{ Comparison of $\| \bar X - X \|$, $R(t^*)$, and  $R(\bar t)$ as a function of $\sigma$ for $\sigma_1=0.7 \sigma$. }
\label{fig2:VE}
\end{figure}
Let us now focus for a moment on the purely denoising case $A=I$. In this case, the projection $\Pi$ is actually the projection onto the subspace where $X$ is defined. While $\bar X$ belongs also to that subspace, this is not true for $Z^t$ in general. Since we do dispose of (an approximation) to $\Pi$ when $A=I$, it would be better to compare
$$
\| \bar X - X \|, \| \Pi Z^{t^*} - X \|, \| \Pi Z^{\bar t} - X \|,
$$
as functions of $\sigma$.
In general, we have
$$
R^\Pi(t)^2 = \| \Pi Z^{t} - X \|^2 = (1+\sigma_1)^2 t^2 -2 t  (1+\sigma_1) + 1,
$$
and
$$
R^\Pi(\bar t) =\frac{\left(\sigma_2^2-\sigma_1
   (\sigma_1+1)^2\right)}{\left(\sigma_2^2+(\sigma_2+1)^2\right)},
$$
 $$
R^\Pi(t^*) =\frac{\sigma_2^2}{\left(\sigma_2^2+(\sigma_1+1)^2\right)}.
$$
The comparison of these quantities is reported in Figure \ref{fig3:VE} and one can observe that, at least for $\sigma\geq 0$, both $\Pi Z^{t^*}$ and $\Pi Z^{\bar t}$ are better approximations of $X$ than $\bar X$.
\begin{figure}[h!]
  \centering
    \includegraphics[width=1\textwidth]{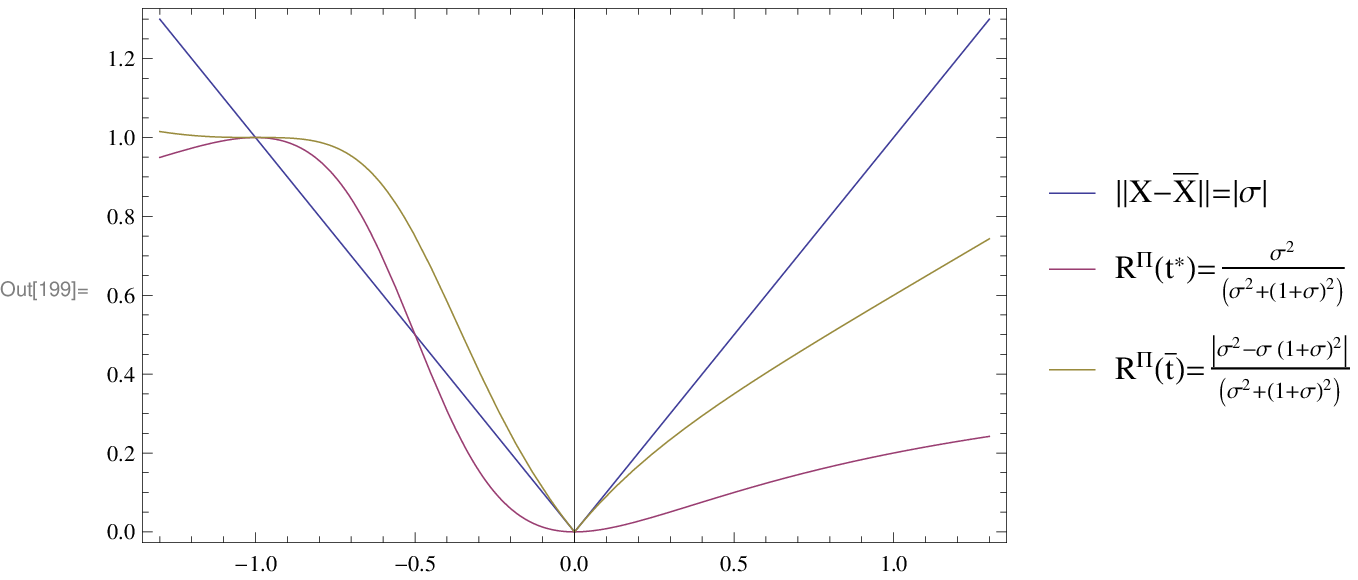}
  \caption{Comparison of $\| \bar X - X \|$, $R^\Pi(\bar t)$, and $R(t^*)$ as a function of $\sigma$ for $\sigma_1=\sigma$. }
\label{fig3:VE}
\end{figure}
\subsection{Empirical projections}

In concrete applications we do not dispose of $\Pi$ but only of its empirical proxy $\widehat \Pi_n$. According to Theorem \ref{thm:QQnbis} they are related by the 
approximation
$$
\| \Pi - \widehat \Pi_n\| \leq \frac{1}{\lambda_{min}} \left ( \sqrt{\frac{m}{n}} + \frac{\tau}{\sqrt{n}} \right ),
$$
with high probability (depending on $\tau>0$), where $m=1$ in our example and $\lambda_{min}=1$ is the smallest
nonzero eigenvalue of $\Sigma_{AX} = A_1 A_1^T$. Hence, we compute
$$
\widehat X = A^{-1} \widehat \Pi_n Y,
$$
and, since $A$ is an orthogonal matrix
\begin{eqnarray*}
\|X- \widehat X  \| &=& \| A^{-1} (\Pi - \widehat \Pi_n) Y\|\\
&=& \|  (\Pi - \widehat \Pi_n) Y\|\\
&\leq& \frac{1+ \tau}{\sqrt{n}} \|A_1 + \sigma W \|\\
&\leq & \frac{1+ \tau}{\sqrt{n}} \sqrt{1+2 \sigma^2 + 2 \sigma_1} \leq  \frac{1+ \tau}{\sqrt{n}} (1 + \sqrt 2 \sigma).
\end{eqnarray*}
Here we used
$$
\| A_1 + \sigma W\|^2 = \| A_1 \|^2+ + \sigma^2 \| W\|^2 +2 \sigma \langle A_1, W \rangle = 1+2 \sigma^2 + 2 \sigma_1.
$$
In view of the approximation relationship above, we can express now 
$$
 \widehat X = (1+\sigma_1+ \epsilon_1, \epsilon_2)^T,
$$
where $\sqrt {\epsilon_1^2 + \epsilon_2^2} \leq   \frac{1+ \tau}{\sqrt{n}} (1 + \sqrt 2 \sigma)$. Then
$$
\| X - \widehat X \| = \sqrt{(\sigma_1+\epsilon_1)^2 + \epsilon_2^2}.
$$
One has also
\begin{eqnarray*}
R_n(t)^2=\| Z^t - \widehat X \|^2 &=& (t(1+ \sigma_1) - (1+ \sigma_1 + \epsilon_1))^2 + (t \sigma_2 - \epsilon_2)^2\\
&=& t^2 ((1+\sigma_1)^2 + \sigma_2^2) - 2 t ( \sigma_2 \epsilon_2 + (1+\sigma_1)\epsilon_1+ (1+ \epsilon_1)^2) + const.
\end{eqnarray*}
Hence, its optimizer is given by
$$
\widehat t_n = \frac{( \sigma_2 \epsilon_2 + (1+\sigma_1)\epsilon_1+ (1+ \epsilon_1)^2)}{(1+\sigma_1)^2 + \sigma_2^2}
$$
and
$$
R(\widehat t_n )=\| Z^{\widehat t_n} - X\| =\sqrt{\frac{( \sigma_2^2(1+\epsilon_2)^2+ 2 \epsilon_2(1+\sigma_1)(\epsilon_1+\sigma_1)\sigma_2+ (1+ \sigma_1)^2(\epsilon_1+\sigma_1)^2)}{(1+\sigma_1)^2 + \sigma_2^2}}
$$

\subsubsection{Concrete example}

We compare in Figure \ref{fig4:VE} the behavior of the difference of the errors $$\|X- \widehat X  \|-\| Z^{\widehat t_n} - X\|$$ depending on $\sigma$, for $\sigma_1=1.3 \sigma $ and $n=100$. In this case, the empirical estimator $Z^{\widehat t_n}$ is a significantly better approximation to $X$ than $\widehat X$ for all noise levels $\sigma \in [-0.07,1/\sqrt{2}]$, while $\widehat X$ keeps being  best estimator, e.g., for $\sigma \in [-1/\sqrt{2},-0.1]$.
\begin{figure}[h!]
  \centering
    \includegraphics[width=1\textwidth]{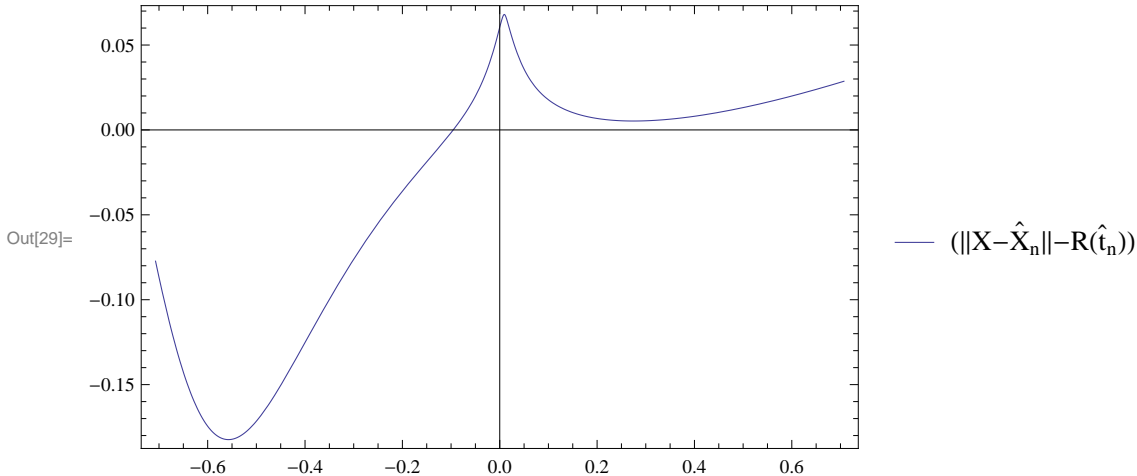}
  \caption{The difference of the errors $\|X- \widehat X  \|-\| Z^{\widehat t_n} - X\|$ as a function of $\sigma$ for $\sigma_1=1.3 \sigma $. For $\sigma \in [-0.07,1/\sqrt{2}]$ one has $\|X- \widehat X  \| > \| Z^{\widehat t_n} - X\|$.}
\label{fig4:VE}
\end{figure}

From these simple two dimensional toy examples and related numerical experiments we deduce the following general principles:
\begin{itemize}
\item The performances of the estimators $\bar X, Z^{t^*}, Z^{\widehat t_n}$ depend very much on how the noise is distributed with respect to the subspace $\mathcal W$ to which $A X$ belongs and its signature; if the noise is not very much concentrated on such subspace, then $\bar X$ tends to be a better estimator, see Fig. \ref{fig2:VE}, while a concentrated noise would make $\bar X, Z^{\widehat t_n}$ rather equivalent and $Z^{t^*}$ the best one for some ranges of noise, see Fig.  \ref{fig0:VE} and  Fig. \ref{fig1:VE}.
\item If the number $n$ of samples is not very large (in the  example above we considered $n=100$), so that the approximation $A^{-1} \widehat \Pi_n Y=\widehat X \approx \bar X=A^{-1}\Pi Y$ would be still too rough, then the estimator $Z^{\widehat t_n}$ may beat $\widehat X$ for significant ranges of noise,  independently of its instance correlation with $\mathcal W$, see Fig. \ref{fig4:VE}.
\end{itemize}
Since it is a priori impossible to know precisely in which situation $\widehat X,  Z^{\widehat t_n}$ perform better as estimators (as it depends on the correlation with $\mathcal W$ of the particular noice instance), in the practice it will be convenient to compute them both, especially when one does not dispose of a large number $n$ of samples.

}

\subsection{The case $A=I$}\label{sec:AisI}
As yet one more example, this time in higher dimension, we consider the simple case
where  $m=d$ and $A=I$, so that $\mathcal W = \mathcal V$.
In this case, we get that
\[
R(t)=-(1-t)X+t \eta.
\]
If $Y\neq 0$,  an easy computation shows that the minimizer of
the reconstruction error $\nor{R(t)}^2$  is 
\begin{equation}
  \label{eq:33}
  t^*=t^*(Y,X)=\varphi\left(\frac{\scal{Y}{X}}{\scal{Y}{Y}}\right),
  \end{equation}
  where
 $$
 \varphi (s)= 
  \begin{cases}
   0 & \text{ if } s \leq 0 \\
   s  & \text{ if } 0<s<1 \\
    1 & \text{ if } s \geq1
  \end{cases}.
$$
If $Y=0$, the solution $Z^t$ does not depend on $t$, so that there is
not a unique optimal parameter and we set $t^*=0$.

We further assume that $X$ is bounded from  $0$ with high probability,
more precisely,
\begin{equation}
  \label{eq:assump}
  {\mathbb P}[\nor{X}< r] \leq 2 \exp\left (-\frac{1}{r^2} \right).
\end{equation}
This assumption is necessary to avoid that the noise is much
  bigger than the signal. 
\begin{theorem}\label{thm:main}
Given $\tau\geq 1$,  with probability greater than $1-\red{6} {\mathrm e}^{-\tau^2}$
\begin{equation}
\abs{\wh{t}-t^*}\leq  
\frac{1 }{\la_{\min}} \left(\sqrt{\frac{d}{n}}+  \frac{\tau}{\sqrt{n}} 
  +\sigma^2\right) + \sigma \ln\left(\frac{\mathrm{e}}{\sigma}\right) (\sqrt{h}+\tau),
\label{eq:38}
\end{equation}
provided that 
 \begin{subequations}
   \begin{align}
     n & \gtrsim (\sqrt{d}+\tau)^2\max\left \{\frac{64}{\la_{\min}^2},1\right \}\label{eq:38a}\\
     \sigma& < \min \left \{\sqrt{\frac{\la_{\min}}{8}}, {\mathrm e}^{1-16\tau^2} \right \}.\label{eq:38b}
   \end{align}
 \end{subequations}

\end{theorem}

\begin{proof}
Without loss of generality, we assume that $\la_{\min}\leq 8$. 
Furthermore, on the event $\set{Y=0}$, by definition $t^*=\wh{t}=0$, so
that we can further assume that $Y\neq 0$.

Since $\varphi$ is a Lipschitz continuous function with Lipschitz constant 1,
\begin{align*}
    \abs{\wh{t}-t^*}&\leq \dfrac{\abs{\scal{Y}{\widehat{X} -X}}}{\nor{Y}^2} 
                      \label{eq:39}\\
                    & \leq  \dfrac{\nor{ (\Pi - \wh{\Pi}) Y -
                      \Pi\eta}}{\nor{Y}}   \nonumber\\ 
& \leq        \nor{ (\Pi - \wh{\Pi}) } + \sigma \dfrac{\nor{\Pi W}}{\nor{Y}}    , \nonumber 
  \end{align*}
where the second inequality is consequence of
~\eqref{eq:47}. Since~\eqref{eq:38a} and~\eqref{eq:38b}
imply~\eqref{cond} and $m=d$, by~\eqref{eq:27} we get
\[
\nor{\wh{\Pi}- \Pi} \lesssim \frac{1 }{\la_{\min}}
 \left(\sqrt{\frac{d}{n}}+  \frac{\tau}{\sqrt{n}} +\sigma^2 \right) \]
 with probability greater than $1 -2 \mathrm{e}^{-\tau^2}$. It is now convenient to denote the probability distribution of $X$ as $\rho_X$, {\it i.e.}, $X \sim \rho_X$.
Fixed $r>0$, set 
\[ \Omega=\set{ \nor{X}< r}\cup\set{ 2 \sigma \scal{X}{W} <-\nor{X}^2/2 },\]
whose probability is bounded by
\begin{alignat*}{1}
  \mathbb P[\Omega] & \leq \mathbb P[\nor{X}<r] + \mathbb P[ 4\sigma \scal{X}{W} <-
  \nor{X}^2 , \red{\nor{X}\geq r}] \\
& =\mathbb P[\nor{X}<r]+\int\limits_{\nor{x}\geq r} \mathbb P[ 4 \sigma \scal{x}{W} < -\nor{x}^2]
  \,d\rho_X(x)\\ 
& \leq\mathbb P[\nor{X}<r]+\int\limits_{\nor{x}\geq r} \exp \left (-\frac{\nor{x}^2}{256\sigma^2} \right)  \,d\rho_X(x) \\
& \leq \mathbb P[\nor{X}<r]+\exp\left (-\frac{r^2}{256\sigma^2} \right),
\end{alignat*}
where we use~\eqref{eq:A3c} with $\xi=-W$ (and the fact that $W$ and
$X$ are independent),  $\tau=\nor{x}/(16\sigma)$
and $\nor{W}_{\psi_2}=1/\sqrt{2}$.  With the choice $r= 16\tau/\ln(\mathrm{e}/\sigma),$ we obtain
\begin{alignat*}{1}
 \mathbb P[\Omega]& \leq \mathbb P[\nor{X}< 16
\tau/\ln(\mathrm{e}/\sigma)]+\exp \left (-\frac{\tau^2}{ \sigma^2\ln^2(\mathrm{e}/\sigma)} \right) \\
& \leq \mathbb P[\nor{X}< 16\tau/\ln(\mathrm{e}/\sigma)]+\exp(-\tau^2) ,
\end{alignat*}
where $\sigma\mapsto \sigma \ln(\mathrm{e}/\sigma)$ is an increasing positive
function on $(0,1]$, so that it is bounded by 1.  Furthermore,  by~\eqref{eq:38b}, {\em i.e.,} $
16\tau/ \ln(\mathrm{e}/\sigma)\leq\frac{1}{\tau}$, we have
\[
\mathbb P[\nor{X}<16\tau/\ln(\mathrm{e}/\sigma)]\leq 
P[\nor{X}<\frac{1}{\tau}]\leq 2 \exp(-\tau^2),
\]
by Assumption~\eqref{eq:assump}.  \red{Hence it holds that 
$\mathbb P[\Omega]\leq 3 e^{-\tau^2}$.}

 On the  event $\Omega^c$
\begin{alignat*}{1}
  \nor{Y}^2 & = \nor{X}^2+2\sigma\scal{X}{ W}+\sigma^2\nor{W}^2 \\
& \geq \nor{X}^2+ 2 \sigma \scal{X}{W}\geq \nor{X}^2/2 \\
& \geq r^2/2 \simeq\tau^2/\ln^2(\mathrm{e}/\sigma)\geq 1/\ln^2(\mathrm{e}/\sigma)
\end{alignat*}
since $\tau\geq 1$.  Finally,~\eqref{eq:57} with $\xi=\Pi W\in\mathcal W$ yields
\[
\nor{\Pi W} \lesssim (\sqrt{h}+\tau) 
\]
with probability greater than $1 - \exp(-\tau^2)$. Taking into account the above estimates, we conclude with probability greater than $1 - \red{4} \exp(-\tau^2)$ that 
$$
\frac{\| \Pi W \|}{\|Y\|} \lesssim \ln{\frac{\mathrm e}{\sigma}} (\sqrt{h} + \tau).
$$
Then, with probability greater than $ 1 - \red{6} \exp(-\tau^2)$, we conclude the estimate
\[
\abs{\wh{t}-t^*} \lesssim \frac{1 }{\la_{\min}}
\left(\sqrt{\frac{d}{n}}+  \frac{\tau}{\sqrt{n}} 
  +\sigma^2\right) + \sigma \ln(\mathrm{e}/\sigma)  (\sqrt{h}+\tau).
\]
\end{proof}
\begin{remark}
  The function $\ln(\mathrm{e}/\sigma)$ can be replaced by any
  positive function $f(\sigma)$ such that $\sigma f(\sigma)$ is an
  infinitesimal function bounded by 1 in the interval $(0,1]$. The
  condition~\eqref{eq:38b} becomes
  \[
  \sigma < \min \left \{\sqrt{\frac{\la_{\min} }{8}}, 1 \right \} \qquad
  f(\sigma)\geq 16 \tau^2,
  \]
  and, if $f$ is strictly decreasing,
  \[
  \sigma < \min\left \{ \sqrt{\frac{\la_{\min} }{8}}, 1,f^{-1}(16
    \tau^2) \right \}.
  \]
\end{remark}
Theorem~\ref{thm:main}  shows that if the number $n$ of examples is large enough and the
noise level is small enough, the estimator $\wh{t}$ is a good
approximation of the optimal value $t^*$. Let us stress very much that the number $n$ of samples
needed to achieve a good accuracy depends at most algebraically on the dimension $d$, more precisely
$n= \mathcal O(d)$. Hence, in this case one does not incur in the {\it curse of dimensionality}.
Moreover, the second term of the error estimate \eqref{eq:38} gets smaller for smaller dimensionality $h$.

\begin{remark}
If there exists an orthonormal  basis $(e_i)_i$ of $\R^d,$ such that the
random variables 
$\scal{W}{e_1},\ldots,\scal{W}{e_d} $ are independent with $\mathbb E[\scal{W}{e_1}^2]=1$, then
\citet[Theorem~2.1]{ruve13}  showed that $\nor{W}$ concentrates around
$\sqrt{d}$ with high probability. Reasoning as in the proof of
Theorem~\ref{thm:main},  by replacing $\nor{X}^2$ with
$\sigma^2\nor{W}^2,$ with high probability it holds that
\[
\abs{\wh{t}-t^*} \lesssim \frac{1 }{\la_{\min}}
\left(\sqrt{\frac{d}{n}}+  \frac{\tau}{\sqrt{n}} 
  +\sigma^2\right) + \frac{1}{\sqrt{d}} (\sqrt{h}+\tau) \, \frac{\tau}{\sqrt{d}}
\]
without assuming condition~\eqref{eq:assump}.
\end{remark}

In Figures \ref{fig0}--\ref{fig0stat}, we show examples of numerical accordance between optimal and estimated regularization parameters. In this case, the agreement between optimal parameter $t^*$ and learned parameter $\widehat t_n$ is overwhelming.

 \begin{figure}[h!]
  \centering
    \includegraphics[width=0.6\textwidth]{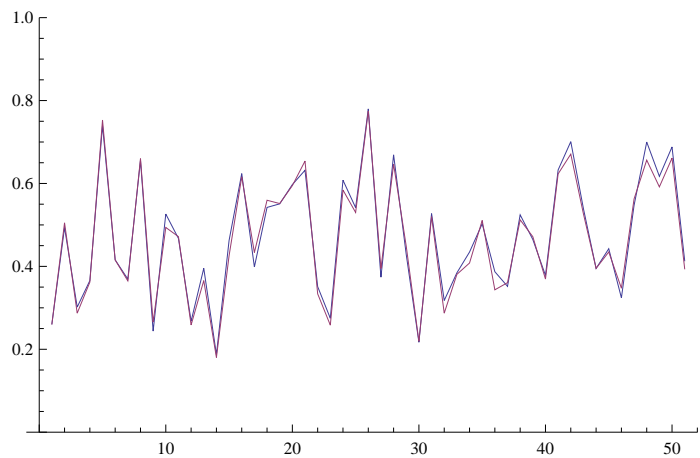}
  \caption{Optimal parameters $t^*$ and corresponding approximations $\widehat t_n$ {\color{black} ($n=1000$)} for $50$ different data $Y = X +\eta$ for $X$ and $\eta$ generated randomly with Gaussian distributions in $\R^d$ for $d=1000$. We assume that $X \in \mathcal V$ for $\mathcal V= \rm{span}\{e_1,\dots,e_5\}$. }
\label{fig0}
\end{figure}

\begin{figure}[h!]
  \centering
    \includegraphics[width=0.4\textwidth]{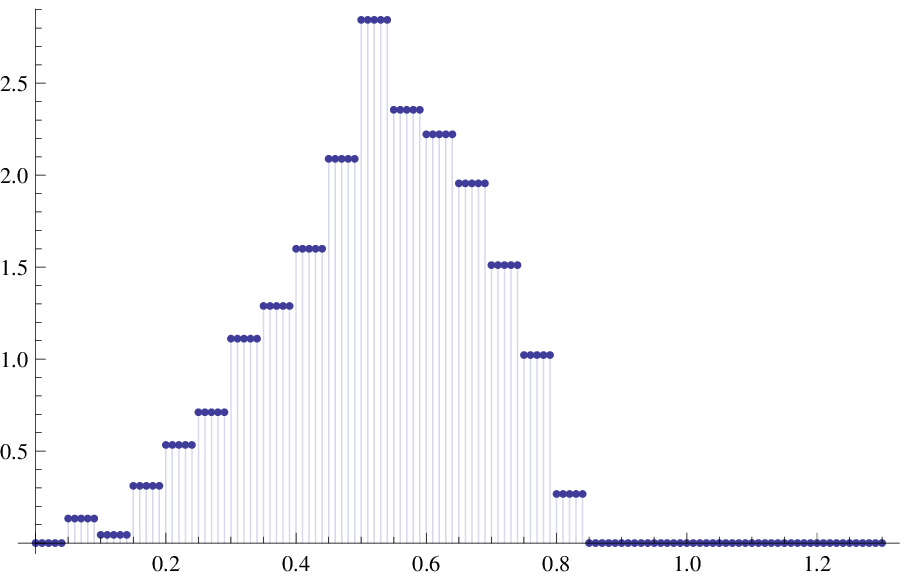}  \includegraphics[width=0.4\textwidth]{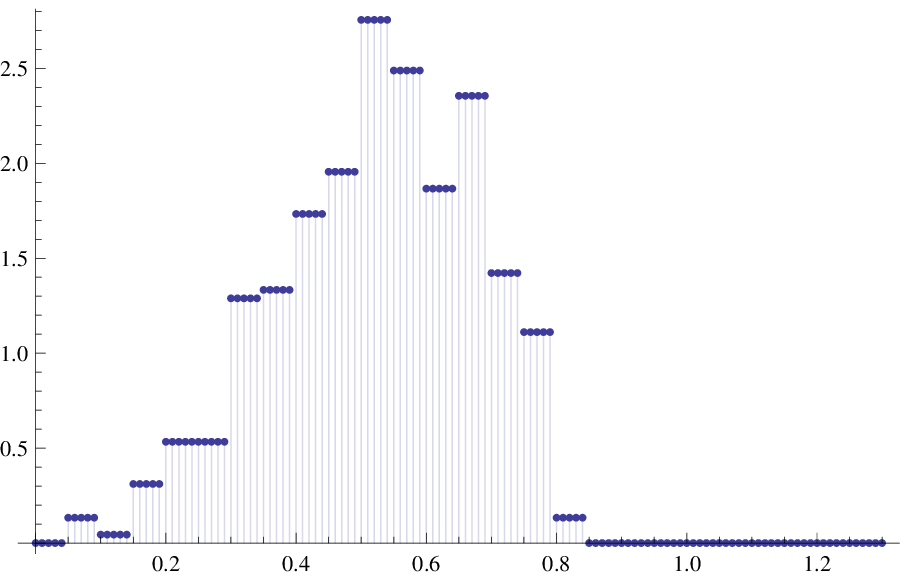}
  \caption{Empirical distribution of the optimal parameters $t^*$ (left) and the corresponding empirical distribution of the learned parameters $\widehat t_n$ (right) for $500$ randomly generated
data $Y = X +\eta$ with the same noise level.}
\label{fig0stat}
\end{figure}

\section{An explicit formula by linearization}\label{sec:5}

While it is not possible to solve the equation $H(t) \approx \widehat H_{n}(t) =0$ by analytic methods for $d>2$ in the general case, one might attempt a linearization of this equation in certain regimes. 
It is well-known that the optimal Tikhonov regularization parameter $\alpha^*=(1-t^*)/t^*$ converges to $0$ for vanishing noise level and this means that   $t^*=t^*(\sigma) \to 1$ as $\sigma \to 0$.
Hence, if the matrix $A$ has a significant spectral gap, \emph{i.e.,} $\sigma_1 \geq \sigma_2 \geq \dots \geq \sigma_d \gg 0$ and $\sigma\approx 0$ is small enough, then 
\begin{equation}\label{eq:50}
\sigma_i \gg (1-t^*),
\end{equation}
and in this case
\begin{equation*}\label{eq:51}
\widehat h_i(t) = \frac{(\sigma_i \widehat \nu_i \widehat \xi_i^{-1}+1)t -1 }{ ((1 - t)+ t\sigma_i^2)^3} \approx \frac{(\sigma_i \widehat \nu_i \widehat \xi_i^{-1}+1)t -1 }{ \sigma_i^6}, \quad t \approx t^*.
\end{equation*}
The above linear approximation is equivalent to replacing $B(t)^{-1}$ with
$B(1)^{-1}=(A^TA)^{-1}$ and
Equation ~\eqref{eq:7} is replaced by the following proxy (at least if $t\approx t^*$)
\begin{align*}
  \wh{H}^{\text{lin}}(t) 
&  =  \scal{ t AA^T(Y- \wh{\Pi} Y) -(1-t) \wh{\Pi} Y }{(AA^T)^{\dagger
    3} Y  }\\
& = \scal{A(A^TA)^{-3}  ( -(1-t) \widehat{X} + t A^T\widehat{\eta})
   }{A \widehat{X} +  \widehat{\eta} ) } \\
 & =  \left( \scal{A (A^TA)^{-3}(\widehat{X} + A^T\widehat{\eta})
   }{A \widehat{X} +  \widehat{\eta}  } \right) t - \scal{A(A^TA)^{-3} \widehat{X}
   }{A \widehat{X} +  \widehat{\eta}  } \nonumber \\
& = \left ( \sum_{i=1}^d  \frac{\widehat \alpha_i}{\sigma_i^5}(\sigma_i
  \widehat \nu_i \widehat \xi_i^{-1}+1) \right ) t - \sum_{i=1}^d  \frac{\widehat
  \alpha_i}{\sigma_i^5} .\nonumber
\end{align*}
The only zero of
$\wh{H}^{\text{lin}}(t)$ is
\begin{alignat*}{1}
 \wh{t}^{\text{lin}}  & =  \dfrac { \scal{\wh{\Pi} Y}{(AA^T)^{\dagger 3} Y  }} 
{\scal{ AA^T(Y- \wh{\Pi} Y) +\wh{\Pi} Y }{(AA^T)^{\dagger 3} Y  } }
\\
& = 1-  \dfrac{ \scal{ Y- \wh{\Pi} Y }{(AA^T)^{\dagger 2} Y  }} 
{\scal{ AA^T(Y- \wh{\Pi} Y) +\wh{\Pi} Y }{(AA^T)^{\dagger 3} Y  }
} \\
& = \dfrac{\scal{A(A^TA)^{-3}
      \widehat{X}}{A \widehat{X} +  \widehat{\eta}  } }{ \scal{A(A^TA)^{-3}(\widehat{X} + A^T\widehat{\eta})   }{A \widehat{X} +  \widehat{\eta}  }    } \\
& = 1- \dfrac{\scal{(AA^T)^{\dagger 2}
      \widehat{\eta}}{A \widehat{X} +  \widehat{\eta}  } }{
    \scal{A(A^TA)^{-3}(\widehat{X} + A^T\widehat{\eta})   }{A
      \widehat{X} +\widehat{\eta}  }    } \\
& =  \frac{ \sum_{i=1}^d \sigma_i^{-5}  \widehat \alpha_i}
{ \sum_{i=1}^d  \sigma_i^{-5} \widehat \alpha_i (\sigma_i \widehat \nu_i
  \widehat \xi_i^{-1}+1)} \\
& =  1- \frac{ \sum_{i=1}^d  \sigma_i^{-4} \widehat{\alpha_i}\widehat{\nu}_i}{
  \sum_{i=1}^d  \sigma_i^{-5} \widehat \alpha_i(\sigma_i \widehat \nu_i
  \widehat \xi_i^{-1}+1)} .
 \end{alignat*}
In Figure \ref{fig3}, we present the comparison between optimal parameters $t^*$ and their approximations $\wh{t}^{\text{lin}}$. Despite the fact that the gap between $\sigma_d$ and $1- t^*$ is not as large as requested in \eqref{eq:50}, the agreement between $t^*$  and   $\wh{t}^{\text{lin}}$ keeps rather satisfactory.  In  Figure \ref{fig0stat}, we report the empirical distributions of the parameters, showing essentially their agreement.
 \begin{figure}[h!]
  \centering
    \includegraphics[width=0.6\textwidth]{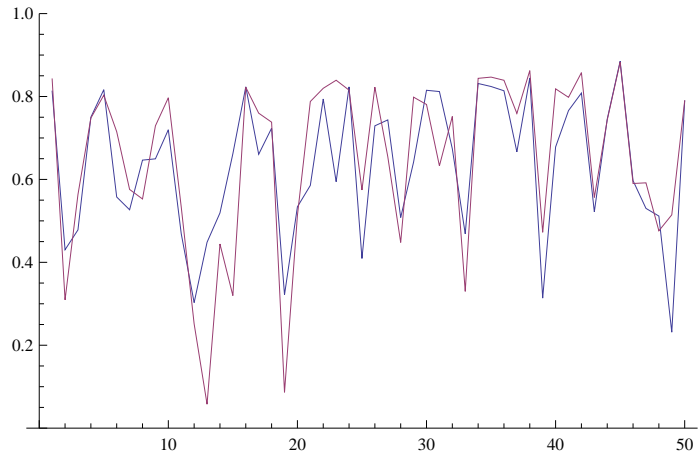}
  \caption{Optimal parameters $t^*$ and corresponding approximations $\wh t$ {\color{black} ($n=1000$)}  for $50$ different data $Y =A X +\eta$ for $X$ and $\eta$ generated as for the experiment of Figure \ref{fig2}. Here we considered a noise level $\sigma=0.006$, so that the optimal parameter $t^*$ can be very close to $0.5$ and the minimal singular value of $A$ is $\sigma_d \approx 0.7$. }
\label{fig3}
\end{figure}
\begin{figure}[h!]
  \centering
    \includegraphics[width=0.4\textwidth]{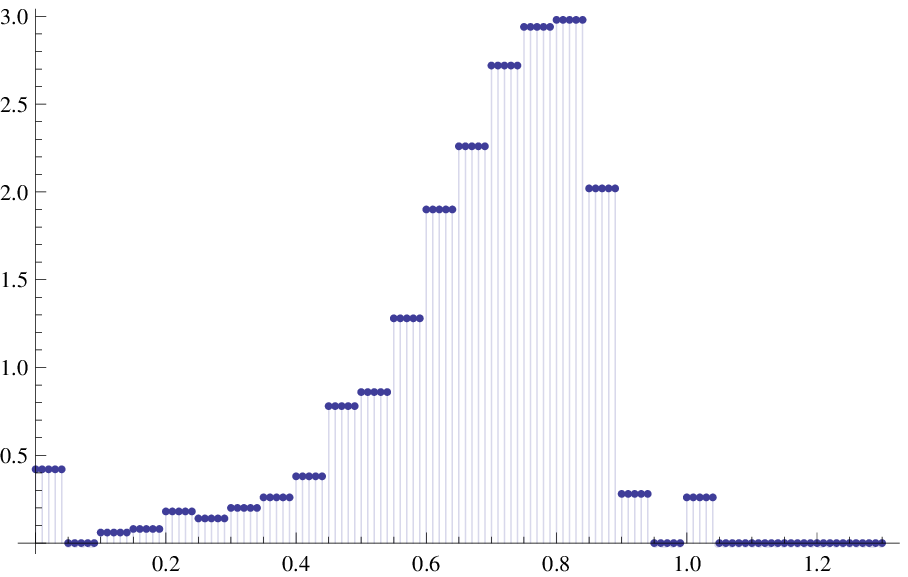}  \includegraphics[width=0.4\textwidth]{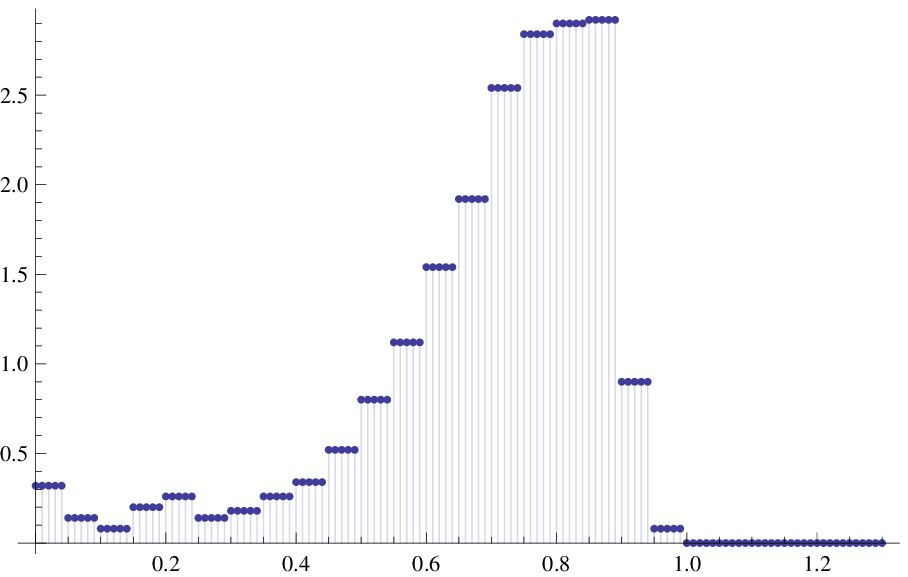}
  \caption{Empirical distribution of the optimal parameters $t^*$ (left) and the corresponding empirical distribution of the approximating parameters $\wh{t}^{\text{lin}}$ (right) for $1000$ randomly generated
data $Y = A X +\eta$ with the same higher noise level. The statistical accordance of the two parameters $t^*$ and $\wh{t}^{\text{lin}}$ is shown.}
\label{fig0stat}
\end{figure}

\section{Conclusions and a glimps to future directions}\label{sec:discussion}

Motivated by the challenge of the regularization and parameter choice/learning relevant for many inverse problems in real-life, in this paper we presented a method to determine the parameter,  based on the usage of a supervised machine learning framework. Under the assumption that the solution of the inverse problem is distributed sub-gaussianly over a small dimensional linear subspace $\mathcal V$ and the noise is also sub-gaussian, we provided a rigorous theoretical justification for the learning procedure of the function, which maps given noisy data into an optimal Tikhonov regularization parameter. We also presented and discussed explicit bounds for special cases and provided techniques for the practical implementation of the method. 
\\

{
Classical empirical computations of  Tikhonov regularization  parameters, for instance the {\it discrepancy principle}, see, e.g., \citep{zbMATH05929140,zbMATH00936298},  may be assuming the knowledge of the noise level $\sigma$ and use regularized Tikhonov solutions for different choices of $t$
\begin{align*}
Z^t & =  t (t A^T A + (1-t)  I)^{-1} A^T Y,
\end{align*}
to return some parameter $\widehat t_\sigma$ which ideally provides $\sigma$-optimal asymptotic behavior of $Z^{\widehat t_\sigma}$ for noise level $\sigma \to 0$. To be a bit more concrete, by using either bisection-type methods or by fixing a pretermined grid of points $\mathcal T=\{ t^{(j)}: j=1,2,\dots\}$, one searches some $t^{(\bar j)} \in \mathcal T$ for which $\| A Z^{t^{(\bar j)} }- Y \| \approx \sigma$ and set $\widehat t_\sigma := t^{(\bar j)}$.
For each $t=t^{(j)}$ the computation of  $Z^t$ has cost of order $\mathcal O(d^2)$ as soon as one has precomputed the SVD of $A$. Most of the empirical estimators require then the  computation of $N_{it}$ instances of $Z^t$ for a total of $\mathcal O(N_{it} d^2)$ operations to return $\widehat t_\sigma$.
\\

Our approach does not assume knowledge of $\sigma$ and it fits into the class of heuristic parameter choice rules, see \citep{zbMATH06288322}; it
requires first to have precomputed the empirical projection $\widehat \Pi_n$ by computation of {\it one single} $h$-truncated  SVD\footnote{One needs to compute $h$ singular values and singular vectors up to the first significant spectral gap.} of the empirical covariance matrix $\widehat \Sigma_n$ of dimension $m \times m$, where $m$ is assumed to be significantly smaller than $d$. In fact, the complexity of methods for computing a truncated SVD out of standard books is $\mathcal O(h m^2)$. When one uses randomized algorithms one could reduce the complexity even to $\mathcal O(\log(h) m^2)$, see  \citep{zbMATH06005534}. Then one needs to solve the optimization problem.
\begin{equation}\label{optparam2}
\min_{t\in [0,1]} \nor{Z^t-\widehat X}^2.
\end{equation}
This optimization is equivalent to finding numerically  roots in $(0,1)$ of the function $\widehat H_n(t)$, which can be performed in the general case and for $d>2$ only by  iterative algorithms. They also require sequential evaluations of $Z^t$ of individual cost $\mathcal O(d^2)$, but they converge usually very fast and they are guaranteed by our results to return $Z^{\widehat t} \approx Z^{t^*}$. If we denote $N_{it}^\circ$ the number of iterations needed for the root finding with appropriate accuracy, we obtain a total complexity of $\mathcal O(N_{it}^\circ d^2)$. In the general case, our approach is going to be more advantageous with respect to classical methods if $N_{it}^\circ \ll N_{it}$. However, 
in the special case where  the noise level is relatively small compared to the minimal positive singular value of $A$, e.g., $\sigma_d \gg \sigma$, then our arguments in Section 5 suggest that one can approximate the relevant root of $\widehat H_n(t)$ and hence $t^*$ with the  smaller cost of $\mathcal O(d^2)$ and no need of several evaluations of $Z^t$. As a byproduct of our results we also showed that 
\begin{equation}\label{hatX}
\widehat X = A^\dagger \widehat \Pi_n Y
\end{equation}
is also a good estimator and this one requires actually only the computation of $ \widehat \Pi_n$ and then the execution of the operations in formula \eqref{hatX} of cost $\mathcal O(m d)$ (if we assumed the precomputation of the SVD of $A$).}\\

Our current efforts are devoted to  the extension of the analysis to the case where the underlying space $\mathcal V$ is actually a smooth lower-dimensional {\it nonlinear} manifold. This extension will be realized by firstly approximating the nonlinear manifold locally on a proper decomposition by means of affine spaces as proposed in \citep{zbMATH06454789} and then applying our presented results on those local {\it linear} approximations.

Another interesting future direction consists of extending the approach to sets $\mathcal V$, unions of linear subspaces as in the case of solutions expressible sparsely the respect to certain dictionaries. In this situation, one would need to consider different regularization techniques and possibly non-convex non-smooth penalty quasi-norms.  

For the sake of providing a first glimps  on the feasibility of the latter possible extension, we consider below the problem of image denoising. In particular, as a simple example of the image denoising algorithm, we consider the wavelet shrinkage \citep{DoJo94}: given a noisy image $Y = X +   \sigma W$ (already expressed in wavelet coordinates), where $\sigma$ is the level of Gaussian noise, the denoised image is obtained by 
$$
 Z^\alpha = \mathbb S_\alpha(X) = \arg\min_{Z} \| Z - Y\|^2 + 2 \alpha \| Z\|_{\ell_1},
$$
where $\| \cdot \|_{\ell_1}$ denotes the $\ell_1$ norm, which promotes a sparse representation of the image with respect to a wavelet decomposition.
Here, as earlier, we are interested in learning the high-dimensional function  mapping  noisy images $X$ into their optimal shrinkage parameters, \emph{i.e.,} an optimal solution of $\| Z^\alpha -  X \|^2 \rightarrow \min_\alpha$. 

Employing a properly modified version of the procedure described in this paper, we are obtaining very exciting and promising results, in particular that  the optimal shrinkage parameter $\alpha$ essentially depends nonlinearly on very few (actually 1 or 2) linear evaluations of $Y$. This is not a new observation and it is a data-driven verification of the well-known results of \citep{DoJo94} and \citep{MR1669536}, establishing that the optimal parameter depends essentially on two meta-features of the noisy image, \emph{i.e.,} the noise level and its Besov regularity. In Figure \ref{fig0num} and Figure \ref{fig1num} we present the numerical results for wavelet shrinkage, which show that our approach chooses a nearly optimal parameter in terms of peak signal-to-noise ratio (PSNR) and visual quality  of the denoising.

\begin{figure}[h!]
  \centering
  {\includegraphics[width=1.1\textwidth]{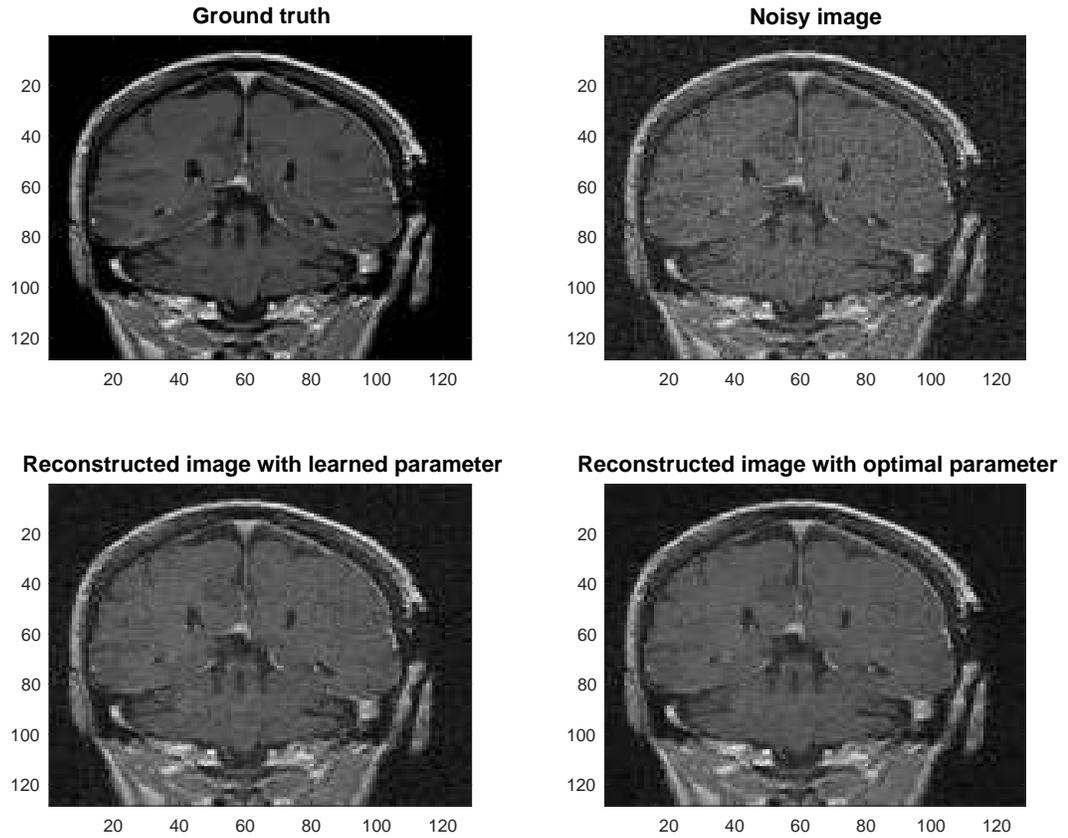}}
  \caption{\red{Numerical Experiments for Wavelet Shrinkage.}}
\label{fig0num}
\end{figure}

\begin{figure}[h!]
  \centering
  {\includegraphics[width=0.9\textwidth]{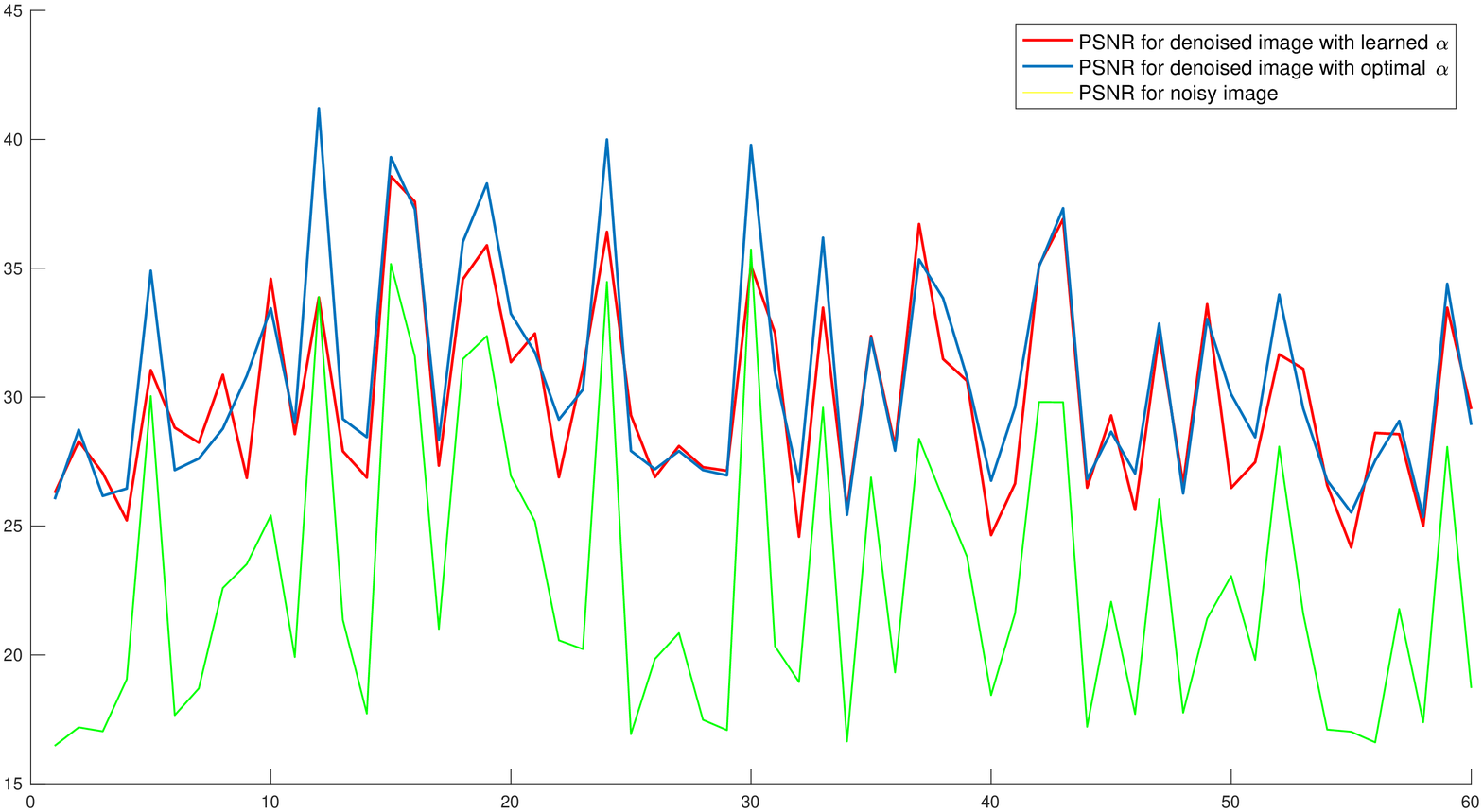}}
  \caption{\red{PSNR between the ground truth image and its noisy version (green line); the ground truth and its denoised version with the optimal parameter (blue line) and the learned parameter (red line). The results are presented for 60 random images.}}
\label{fig1num}
\end{figure}

\appendix

\acks{E. De Vito is a member of the Gruppo Nazionale per l'Analisi
  Matematica, la Probabilit\`a e le loro Applicazioni (GNAMPA) of the
  Istituto Nazionale di Alta Matematica (INdAM). M. Fornasier
  acknowledges the financial support of the ERC-Starting Grant
  HDSPCONTR ``High-Dimensional Sparse Optimal Control''. V. Naumova
  acknowledges the support of project ``Function-driven Data Learning
  in High Dimension'' (FunDaHD) funded by the Research Council of Norway.} 



\appendix

\section{Perturbation result for compact operators}
We recall the following perturbation result for compact operators in Hilbert
spaces \citep{ans71} and  \citep[Theorem~3]{zwabla06}, whose proof also
holds without the assumption that the spectrum is simple \citep[see][Theorem~20]{robede10}.
\begin{proposition}\label{propro}
Let  $\mathcal A$  and $\mathcal B$ be two compact positive operators on a Hilbert space $\hh$ and
denote by $(\alpha_j)_{j=1}^N$ and $(\beta_\red{\ell})_{\red{\ell}=1}^L$ the
corresponding families of
(distinct) strictly positive eigenvalues of $\mathcal A$ and $\mathcal B$ ordered in a decreasing way.  
For all $1\leq j\leq N$, denote by $P_j$ (resp. $Q_\red{\ell}$ with $1\leq \red{\ell}\leq L$) the projection onto the vector space
spanned by the eigenvectors of $\mathcal A$ (resp. $\mathcal B$) whose eigenvalues are
greater or equal than $\alpha_j$ (respect. $\beta_\red{\ell}$). Let $j\leq N$ such that $\nor{\mathcal A-\mathcal B}<\frac{\alpha_j- \alpha_{j+1}}{4}$, then  
there exists $\red{\ell}\leq L$ so that
\begin{align}
& \beta_{\red{\ell}+1}< \frac{\alpha_j+ \alpha_{j+1}}{2} <\beta_\red{\ell}\nonumber\\  
 & \nor{Q_\red{\ell}- P_{\red{j}}} \le \frac{2}{\alpha_j- \alpha_{j+1}}
 \nor{\mathcal A-\mathcal B}\label{eq:A1}\\
& \dim{Q_\red{\ell}\hh}=\dim{P_j\hh} .\nonumber
\end{align}
If $\mathcal A$ and $\mathcal B$ are Hilbert-Schmidt,
the operator norm in the above bound  can be replaced by the
Hilbert-Schmidt norm.
\end{proposition}
In the above proposition, if $N$ or $M$ are finite, $\alpha_{N+1}=0$
or  $\beta_{L+1}=0$. \red{We note that, if the eigenvalues of $A$ or
  $B$ are not simple, in  general $\ell\neq j$. However, the above
  result shows that, given $j=1,\ldots,N$ there  exists a unique
  $\ell=1,\ldots,L$ such that  $\dim{Q_\red{\ell}\hh}=\dim{P_j\hh}$
  and $\beta_\ell$ is the smallest eigenvalue of $B$ greater than $(\alpha_j+\alpha_{j+1})/2$.}

\section{Sub-gaussian vectors}\label{sec:sub-gaussian-vectors}

  We recall some facts about sub-gaussian random vectors and we follow
  the presentation in \citep{ver12}, which provides the proofs of the
  main results in Lemma 5.5. 

\begin{subequations}
  \begin{proposition}
 Let $\xi$ be a sub-gaussian random vector in $\R^d$. \red{Then, for all
 $\tau>0$ and $v\in\R^d$}
\begin{equation}
  \label{eq:A3}
  \mathbb P[\abs{\scal{\xi}{v}}>3 \nor{\xi}_{\psi_2}\nor{v} \tau]\leq 2
    \exp(-\tau^2).
\end{equation}
Under the further assumption that $\xi$ is centered,  then \red{
\begin{alignat}{1}
    \label{eq:A4}
    \mathbb E[\exp(\tau\scal{\xi}{v})] & \leq
    \exp\left(8 \tau^2\nor{\xi}^2_{\psi_2}\right) \\
  \label{eq:A3c}
  \mathbb P[\scal{\xi}{v}>4\sqrt{2} \nor{\xi}_{\psi_2}\nor{v} \tau]& \leq 
    \exp(-\tau^2).
\end{alignat}}
  \end{proposition}
\end{subequations}
  \begin{proof}
 We  follow  the idea in \citet[Lemma 5.5]{ver12} of explicitly
 computing the constants. By rescaling $\xi$ to
 $\xi/\nor{\xi}_{\psi_2}$, we can assume that
 $\nor{\xi}_{\psi_2}=1$. 

Let $c>0$ be a small constant to be fixed.
Given, $v\in S^{d-1}$,  set $\chi=\scal{\xi}{v}$,
which is a real sub-gaussian vector. By Markov inequality,
\begin{align*}
  \mathbb P[\abs{\chi}>\tau] & = \mathbb
  P[c\abs{\chi}^2>c\tau^2] = \mathbb
  P[\exp(c\abs{\chi}^2)>\exp(c\tau^2)] \\
&  \leq  \mathbb E[ \exp(c\abs{\chi}^2)] {\mathrm e}^{-c \tau^2}.
\end{align*}
\red{ By~\eqref{eq:A2a}, we get that $\mathbb E[ \abs{\chi}^q]\leq
  q^{q/2}$, so that} 
\begin{align*}
\mathbb E[ \exp(c\abs{\chi}^2)] = 1 + \sum_{k=1}^{+\infty}
  \frac{c^k}{k!} \mathbb E[\abs{\chi}^{2k}] \leq  1 +
  \sum_{k=1}^{+\infty}\frac{(2ck)^k}{k!} \leq 1 + \frac{1}{\mathrm{e}}   \sum_{k=1}^{+\infty} (2c e)^k= 1 +
  \frac{2c}{1-2c \mathrm{e}}, 
\end{align*}
where we use the estimate $k!\geq \mathrm{e} (k/\mathrm{e})^k$ for
$k\ge 1$. Setting $c=1/9$,  $1 + \frac{2c}{1-2c \mathrm{e}} <2$, so that
\[
\mathbb P[\abs{\chi}>3\tau]\leq 2 \exp(-\tau^2),
\] 
\red{ so that~\eqref{eq:A3} is proven.

Assume now that $\mathbb E[\xi]=0$. By~(5.8) in \citet{ver12}
  \begin{equation}
\mathbb E[\exp(\frac{\tau}{\mathrm e}\chi)] \leq 1 +
\sum_{k=2}^{+\infty}
\left(\frac{\abs{\tau}}{\sqrt{k}}\right)^k,\label{eq:16}
\end{equation}
and, by~(5.9) in \citet{ver12},

  \begin{equation}
\exp\left(\tau^2C^2\right)\geq 1 + \sum_{h=1}^{+\infty}
\left(\frac{C\abs{\tau}}{\sqrt{h}}\right)^{2h}.\label{eq:18}
\end{equation}
If $\abs{\tau}\leq 1$, fix an even $k\geq 2$ so that $k=2h$ with
$h\geq 1$. The $k$-th and
$(k+1)$-th terms of the series~\eqref{eq:16} is 
\[
\left(\frac{\abs{\tau}}{\sqrt{k}}\right)^k +
\left(\frac{\abs{\tau}}{\sqrt{k+1}}\right)^{k+1}\leq 2
\left(\frac{\abs{\tau}}{\sqrt{k}}\right)^k = \frac{2}{ C^{2h} 2^h}
\left(\frac{\abs{\tau}}{\sqrt{2h}}\right)^{2h} \leq
\left(\frac{C\abs{\tau}}{\sqrt{k}}\right)^k,
\]  
where the right hand side is the $h$-th term of the
series~\eqref{eq:18} and the last inequality holds true if $C\geq
1$. Under this assumption
  \begin{equation}
\mathbb E[\exp(\frac{\tau}{\mathrm e}\chi)] \leq
\exp\left(\tau^2C^2\right)\qquad |\tau|\leq 1.\label{eq:19}
\end{equation}
If $\abs{\tau}\geq 1$, fix an odd $k\geq 3$, so that $k=2h-1$ with
$h\geq 2$, then the $k$-th and $(k+1)$-th terms of the
series~\eqref{eq:16} is 
\[
\left(\frac{\abs{\tau}}{\sqrt{k}}\right)^k +
\left(\frac{\abs{\tau}}{\sqrt{k+1}}\right)^{k+1}\leq 2
\left(\frac{\abs{\tau}}{\sqrt{k}}\right)^{k+1} =   \left(\frac{\sqrt{2}h}{C^2(2h-1)}  \right)^{h}
\left(\frac{\abs{\tau}}{\sqrt{h}}\right)^{2h}\leq \left(\frac{\abs{\tau}}{\sqrt{h}}\right)^{2h},
\] 
where the right hand side is the $h$-th term of the
series~\eqref{eq:18} and the inequality holds true provided that
\[
 C^2 \geq \sup_{h\geq 2} \frac{\sqrt{2}h}{(2h-1)} = \frac{2\sqrt{2}}{3},
\]
which holds true with the choice $C=1$. Furthermore, the term
of~\eqref{eq:16} with $k=2$ is clearly bounded by  term
of~\eqref{eq:18} with $h=1$, so that we get
\[
\mathbb E[\exp(\frac{\tau}{\mathrm e}\chi)] \leq
\exp\left(\tau^2\right)\qquad |\tau|\geq 1.
\]
Together  with bound~\eqref{eq:19} with $C=1$, the abound bound
gives~\eqref{eq:A4} since $e^2<8$.

Finally, reasoning as in the proof of~\eqref{eq:A3} and by using~\eqref{eq:A4}
\begin{align*}
  \mathbb P[\chi>\tau] & =   P[\exp(c\chi)>\exp(c\tau)] \\
&  \leq  \mathbb E[ \exp(c\chi )] {\mathrm e}^{-c \tau} \\
& \leq  \exp(8c^2-c\tau),
\end{align*}
which takes the minimum at $c=\tau/16$. Hence,
\[
\mathbb P[\chi>\tau]  = \exp(-\frac{\tau^2}{32}).
\]
}
  \end{proof}
  \begin{remark}
Both~\eqref{eq:A3} and~\eqref{eq:A4} (for a suitable constants
instead of $\nor{\xi}_{\psi_2}$)    are sufficient conditions for
sub-gaussianity and~\eqref{eq:A4} implies that $\mathbb E[\xi]=0$, see
\cite[Lemma 5.5]{ver12}.
  \end{remark}

The following proposition bounds the Euclidean  norm of a sub-gaussian
vector. The proof is standard and essentially based on the results
in~\citet{ver12},  but we were  not able to find the precise reference. 
The centered case is done in~\citet{rigolet15}.

\begin{proposition}
 Let $\xi$ a sub-gaussian random vector in $\R^d$.  Given $\tau>0$
 with probability greater than $1 -2 \mathrm{e}^{-\tau^2}$
 \begin{equation}
 \nor{\xi} \leq  3\nor{\xi}_{\psi_2}  ( \sqrt{\red{7} d} + 2\tau )\leq 9\nor{\xi}_{\psi_2}  ( \sqrt{ d} + \tau ). \label{eq:56}
 \end{equation}
If $\mathbb E[\xi]=0$, then  with probability greater than $1 - \mathrm{e}^{-\tau^2}$
\begin{equation}
  \label{eq:57}
  \nor{\xi} \leq 8\nor{\xi}_{\psi_2}  ( \sqrt{\red{\frac{7}{2}} d} + \sqrt{2} \tau ) \leq 16\nor{\xi}_{\psi_2}  ( \sqrt{d} + \tau ). 
\end{equation}
\end{proposition}
\begin{proof}
As usual we assume that $ \nor{\xi}_{\psi_2}=1$.
Let $\mathcal N$ be a $1/2$-net of $S^{d-1}$. Lemmas~5.2 and~5.3  in
\citet{ver12} give
\[
\nor{\xi} \leq 2 \max_{v\in\mathcal N} \scal{\xi}{v} \qquad
\abs{\mathcal N} \leq 5^d.
\]  
Fixed $v\in \mathcal N$,~\eqref{eq:A3} gives that 
\[ \mathbb P[\abs{\scal{\xi}{v}}>3 t ]\leq 2
    \exp(-t^2).\]
By union bound
\begin{align*}
\mathbb P[\nor{\xi} >6 t] & \leq   \mathbb P[\max_{v\in\mathcal N}
                               \abs{\scal{\xi}{v}}>3 t] 
 \leq 2 \abs{\mathcal N} \exp(-t^2) \leq 2 \exp(\red{d\ln 5} -t^2).
\end{align*}
\red{Bound~\eqref{eq:56} follows  with the choice $t=\red{\tau+
    \sqrt{7d/4}}$ taking into account that $t^2-\red{d\ln 5}>\tau^2$
 since $\ln 5<7/4$.}

Assume that $\xi$ is centered and use~\eqref{eq:A3c} instead
of~\eqref{eq:A3}. Then
\begin{align*}
\mathbb P[\nor{\xi} >8\sqrt{2} t] & \leq   \mathbb P[\max_{v\in\mathcal N}
                               \abs{\scal{\xi}{v}}>\red{4\sqrt{2}} t]
 \leq \abs{\mathcal N} \exp(-t^2) \leq \exp(\red{d\ln 5} -t^2).
\end{align*}
\red{As above, the choice  $t=\red{\tau+ \sqrt{7d/4}}$ provides the bound~\eqref{eq:57}.}
\end{proof}
\begin{remark}
Compare with~Theorem~1.19 in~\citet{rigolet15},  noting that by~\eqref{eq:A4}
the parameter $\sigma$ in Definition~1.2
of~\citet{rigolet15}  is bounded by $4\nor{\xi}_{\psi_2}$.
\end{remark}

The following result is a concentration inequality for the second
momentum of sub-gaussian random vector, see Theorem 5.39 and Remark~5.40 in
\citet{ver12} and footnote~20. 
\begin{theorem}\label{thm:concen}
Let $\xi\in\R^d$ be a sub-gaussian vector random vector  in $\R^d$. Given
a family $\xi_1,\ldots,\xi_n$ of random 
vectors independent and identically distributed as $\xi$, then for 
$\tau>0$ 
\begin{equation*}
  \label{eq:A5}
  \mathbb P[\nor{\frac{1}{n}\sum_{i=1}^n \xi_i\otimes\xi_i -\mathbb
    E[\xi \otimes \xi]}>\max\set{\delta,\delta^2}]\leq  2\mathrm{e}^{-\tau^2}
\end{equation*}
where
\[
\delta=C_{\xi} \left(\sqrt{\frac{d}{n}}+ \frac{\tau}{\sqrt{ n}}\right),
\]
and $C_{\xi} $ is a constant depending only on the sub-gaussian norm
$\nor{\xi}_{\psi_2}$.
\end{theorem}

\vskip 0.2in
\bibliography{biblio}

\end{document}